%% file: main.tex
\documentclass[final,hidelinks,onefignum,onetabnum]{siamart250211}

\usepackage{lipsum}
\usepackage{amsfonts}
\usepackage{epstopdf}
\usepackage{url}
\usepackage{afterpage}
\ifpdf
  \DeclareGraphicsExtensions{.eps,.pdf,.png,.jpg}
\else
  \DeclareGraphicsExtensions{.eps}
\fi

\usepackage[normalem]{ulem}
\usepackage[labelfont=bf, size=small]{caption}
\usepackage[labelformat=empty, size=small]{subcaption}

\usepackage{microtype, fancyhdr, siunitx, graphicx, hyperref, epstopdf, tikz,
  adjustbox, mathtools, float, anyfontsize, relsize}
\usepackage{bm, amssymb, amsmath, mathrsfs, dsfont}
\usepackage{etoolbox}
\usepackage[algo2e, ruled, linesnumbered, resetcount]{algorithm2e}
\SetFuncSty{textsc} \SetKwProg{Function}{}{:}{}
\SetKwComment{Comment}{$\triangleright$\ }{}

\usetikzlibrary{backgrounds}
\usetikzlibrary{decorations.pathreplacing}
\usetikzlibrary{patterns}
\usetikzlibrary{math}

\usepackage{enumitem}
\setlist[itemize]{leftmargin=2em, topsep=0.5\baselineskip, itemsep=0.5\baselineskip}

\input{preamble.tex}

\def\spacingset#1{\renewcommand{\baselinestretch}%
{#1}\small\normalsize} \spacingset{1}

\newsiamremark{remark}{Remark}
\newsiamremark{hypothesis}{Hypothesis}
\crefname{hypothesis}{Hypothesis}{Hypotheses}
\newsiamthm{claim}{Claim}
\newsiamthm{conjecture}{Conjecture}

\headers{A Butterfly-accelerated Manifold Harmonic Transform}{Beckman, Potter, and  O'Neil}

\title{A Butterfly-accelerated Manifold Harmonic Transform\thanks{Submitted to
    the editors DATE. \funding{P. G. Beckman was partially supported by the
    Office of Naval Research under award~\#N00014-21-1-2383 and by the U.S.
    Department of Energy, Office of Science, Office of Advanced Scientific
    Computing Research, Department of Energy Computational Science Graduate
    Fellowship under Award Number DE-SC0022158. S. F. Potter was partially
    supported by the Office of Naval Research under award~\#N00014-18-1-2307. M.
    O'Neil was partially supported by the Office of Naval Research under
    awards~N00014-18-1-2307 and~\#N00014-21-1-2383.} } }

\author{Paul G. Beckman\thanks{Oden Institute, University of Texas at Austin,
        Austin, TX (\texttt{\email{paul.beckman@austin.utexas.edu}})}
  \and
  Samuel F. Potter\thanks{Robert McNeel \& Associates, Seattle, WA
    (\texttt{\email{sam@mcneel.com}})}
  \and
  Michael O'Neil\thanks{Courant Institute, New York University, New York NY
    (\texttt{\email{oneil@cims.nyu.edu}})}. }

\usepackage{amsopn}

\ifpdf
\hypersetup{ pdftitle={A Butterfly-accelerated Manifold Harmonic Transform}
  pdfauthor={P. G. Beckman, S. F. Potter, and M. O'Neil} }
\fi

\begin{document}

\maketitle

\begin{abstract}
  The eigenfunctions of the Laplacian are a natural basis of functions for many
  tasks in computational mathematics. On the circle and sphere, the
  eigenfunctions are given by complex periodic exponentials and spherical
  harmonics, respectively, and much work has been done to develop fast
  algorithms for analyzing and synthesizing data in these bases. In this work,
  we generalize these special-case transforms to Laplace-Beltrami eigenfunctions
  of arbitrary surfaces, referred to as manifold harmonics. The resulting fast
  algorithm for computing linear combinations of the manifold harmonics is based
  on a butterfly factorization, which hierarchically compresses the transform
  matrix by constructing nested low-rank approximations of carefully selected
  submatrices. Several numerical examples are provided which demonstrate the
  speedups and reduction in memory requirements achieved by our algorithm for a
  variety of geometries, discretizations, and applications. In addition, a
  detailed analysis of the algorithm is given in the case that the underlying
  manifold is the flat periodic square.
\end{abstract}



\section{Introduction}
\label{sec:intro}

The Fast Fourier Transform (FFT)~\cite{van1992computational} and Nonuniform Fast
Fourier Transform (NUFFT)~\cite{dutt1993fast,greengard2004accelerating} have
revolutionized computational methods across fields including signal processing,
differential equations, and statistics by accelerating the computation of
discrete trigonometric sums with~$n$ terms and~$n$ outputs from~$\bO(n^2)$
to~$\cO (n \log n)$ time. In the case of the FFT, this acceleration is achieved
purely by algebraic manipulations of the quantity~$e^{2\pi i j k/n}$. On the
other hand, the NUFFT relies on Fourier convolution theorems and numerical
approximations based on an FFT applied to a periodized kernel function. However,
both of these algorithms strongly rely on the fact that the underlying geometry
is simple and periodic; spectral analysis in~$\bbR^d$ or~$\mathbb{T}^d$ need
only manipulate complex exponentials.

In order to perform spectral analysis in more general domains, one requires a
generalization of the Fourier transform which takes into account the underlying
geometry. To motivate such a generalization, consider the two-dimensional
type-II nonuniform discrete Fourier transform:
\begin{align} \label{eq:2D-DFT}
  f(\bm{x}_j) = \sum_{\bm{k} \in [-\frac{\sqrt{m}}{2},\frac{\sqrt{m}}{2}-1]^2}
  c_{\bm{k}} \, e^{i\bm{k}\cdot \bm{x}_j}, \quad \text{for } j=1,\dots,n.
\end{align}
The above sum evaluates the function~$f$ at $n$ nonuniform spatial locations
$\bm{x}_j \in [-\pi,\pi]^2$ given the coefficients~$c_{\bm{k}}$, which denote
Fourier coefficients for $m$ uniform integer frequencies $\bm{k} \in \Z^2$.
Noting that the functions~$\phi_{\bm{k}}(\bm{x}) = e^{i\bm{k}\cdot \bm{x}}$ are
exactly the eigenfunctions of the Laplacian $\Delta$ on $[-\pi,\pi]^2$ under
periodic boundary conditions, it is natural to consider a generalization of the
Fourier transform which uses Laplacian eigenfunctions to perform harmonic
analysis on domains other than the flat periodic square.

For a general compact $d$-manifold $\M$ with boundary~$\partial\M$, consider the
Laplace-Beltrami eigenproblem
\begin{equation}
  \label{eq:eigenproblem}
  \begin{aligned}
    \LBO \phi &= \lambda \phi, & \quad &\text{in } \M, \\
    g\Big(\phi, \pder{\phi}{n}\Big) &= 0 & \quad &\text{on } \partial\M,
  \end{aligned}
\end{equation}
where~$\LBO$ is the Laplace-Beltrami operator and $g(\phi, \partial
\phi/\partial n)$ represents any suitable set of boundary conditions, e.g.
Dirichlet, Neumann, Robin, or a mixture thereof. Of course if~$\M$ does not have
a boundary, then the above eigenproblem needs no boundary conditions. This
eigenproblem yields a set of nonnegative eigenvalues $0 \leq \lambda_1 \leq
\lambda_2 \leq \dots$ whose corresponding eigenfunctions~$\phi_k$ form an
orthonormal basis for $L^2(\M)$~\cite{chavel1984eigenvalues}. A generalization
of the discrete Fourier transform~\eqref{eq:2D-DFT} to the manifold~$\M$ is then
given by the discrete \textit{manifold harmonic transform} (MHT)
\begin{equation}
  \label{eq:MHT-sum}
  f(\bm{x}_j) = \sum_{k=1}^m c_k \, \phi_k(\bm{x}_j), \quad \text{for } j=1,\dots,n,
\end{equation}
which we write in the equivalent matrix form $\vct{f} = \mtx{\Phi} \vct{c}$
where~$\mtx{\Phi} \in \C^{n \times m}$, \mbox{$\mtx{\Phi}_{jk} =
\phi_k(\bm{x}_j)$}, and~$\vct{f}_j = f(\bm{x}_j)$.
Figure~\ref{fig:eigenfunctions} shows some numerically computed Laplace-Beltrami
eigenfunctions~$\phi_k$ on an example manifold $\M$, each of which serves as a
column of the matrix~$\mtx{\Phi}$. Note that as we increase the index $k$, the
eigenfunctions exhibit oscillations on decreasing lengthscales, providing a
natural multiscale Fourier-like basis on~$\M$.

\begin{figure}
  \centering
  \newcommand\wdth{0.15\textwidth}
  \includegraphics[trim={6.5cm, 1.5cm, 6cm, 1.1cm}, clip, width=\wdth]{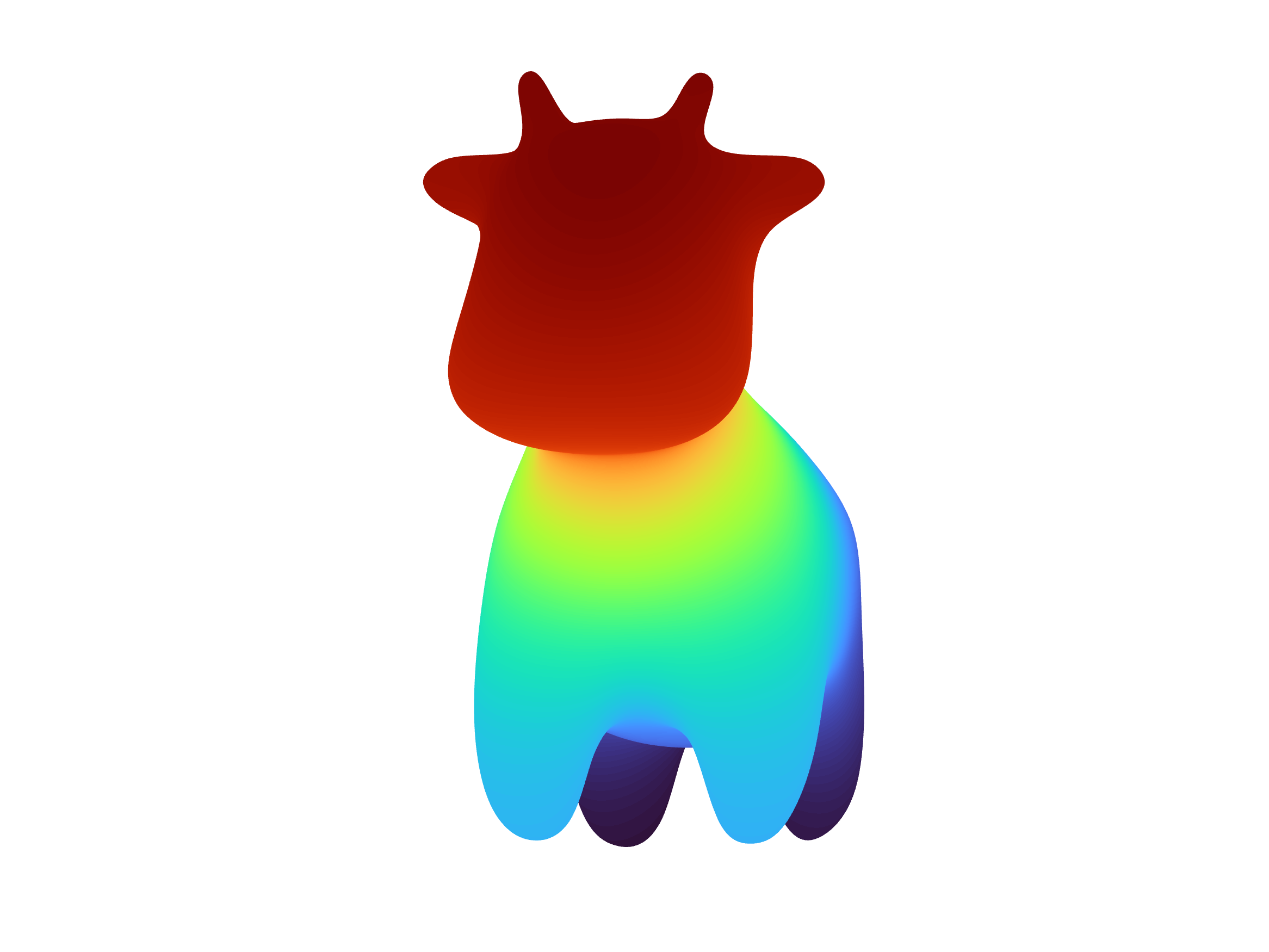}
  \includegraphics[trim={6.5cm, 1.5cm, 6cm, 1.1cm}, clip, width=\wdth]{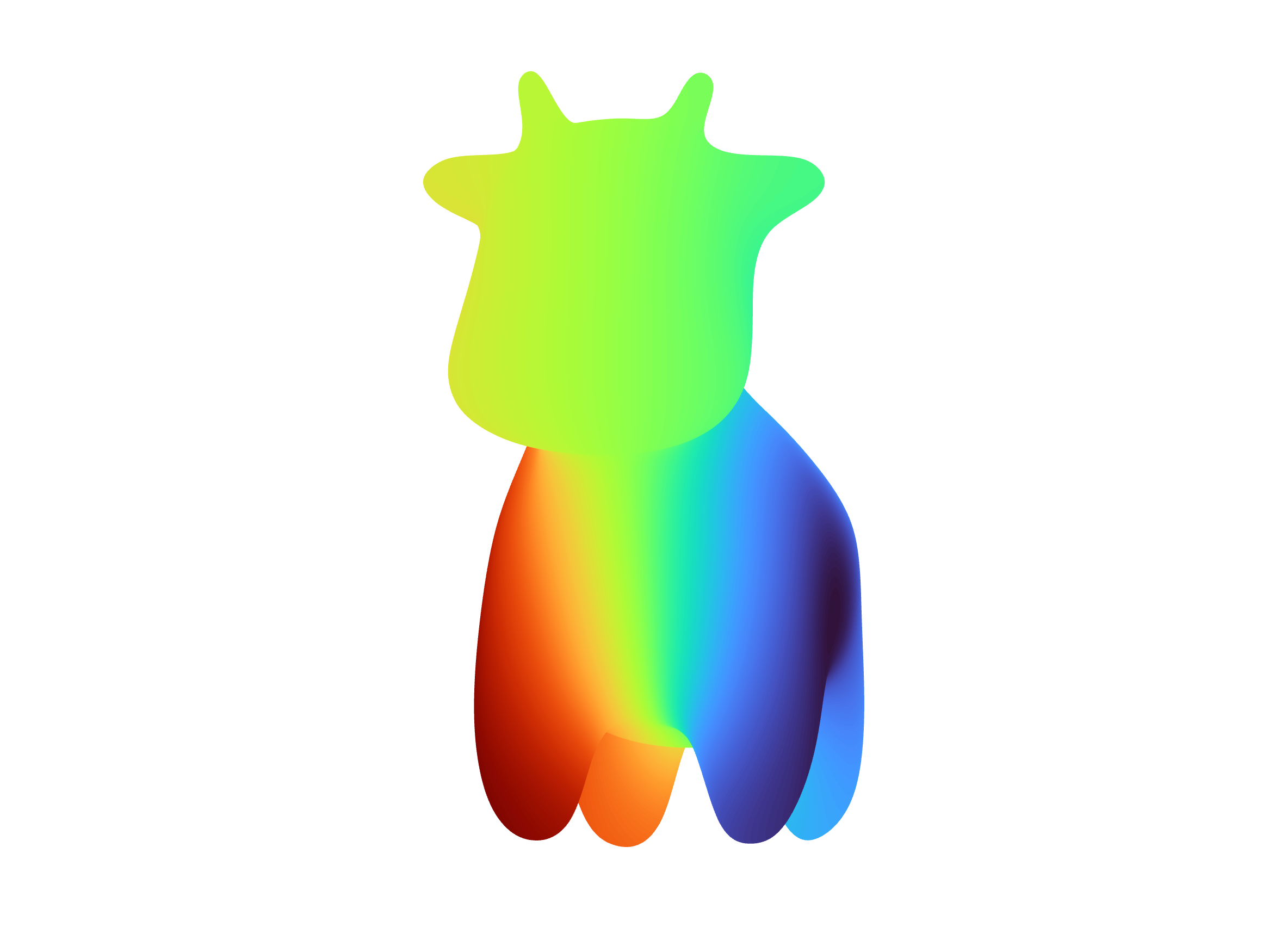}
  \includegraphics[trim={6.5cm, 1.5cm, 6cm, 1.1cm}, clip, width=\wdth]{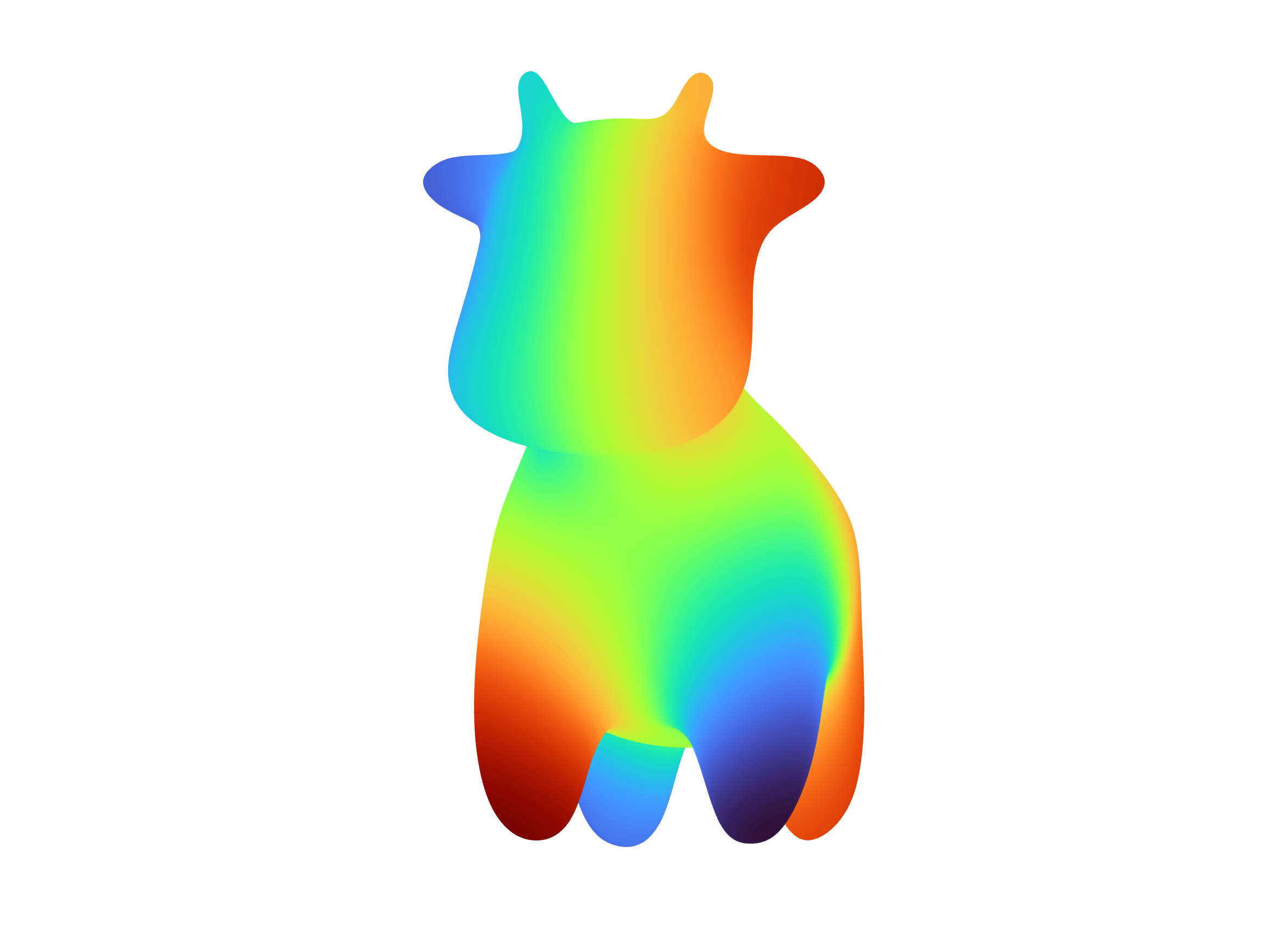}
  \includegraphics[trim={6.5cm, 1.5cm, 6cm, 1.1cm}, clip, width=\wdth]{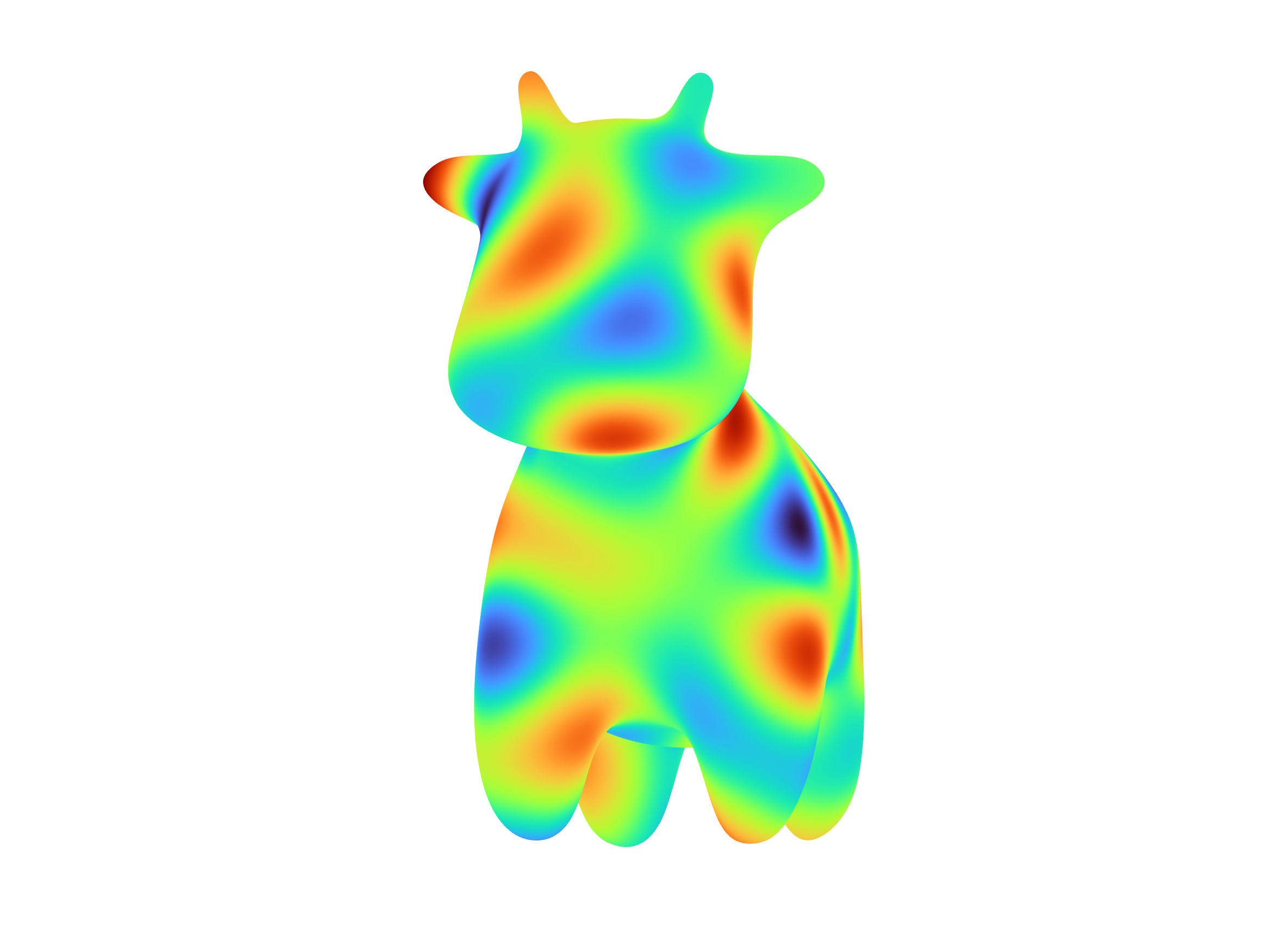}
  \includegraphics[trim={6.5cm, 1.5cm, 6cm, 1.1cm}, clip, width=\wdth]{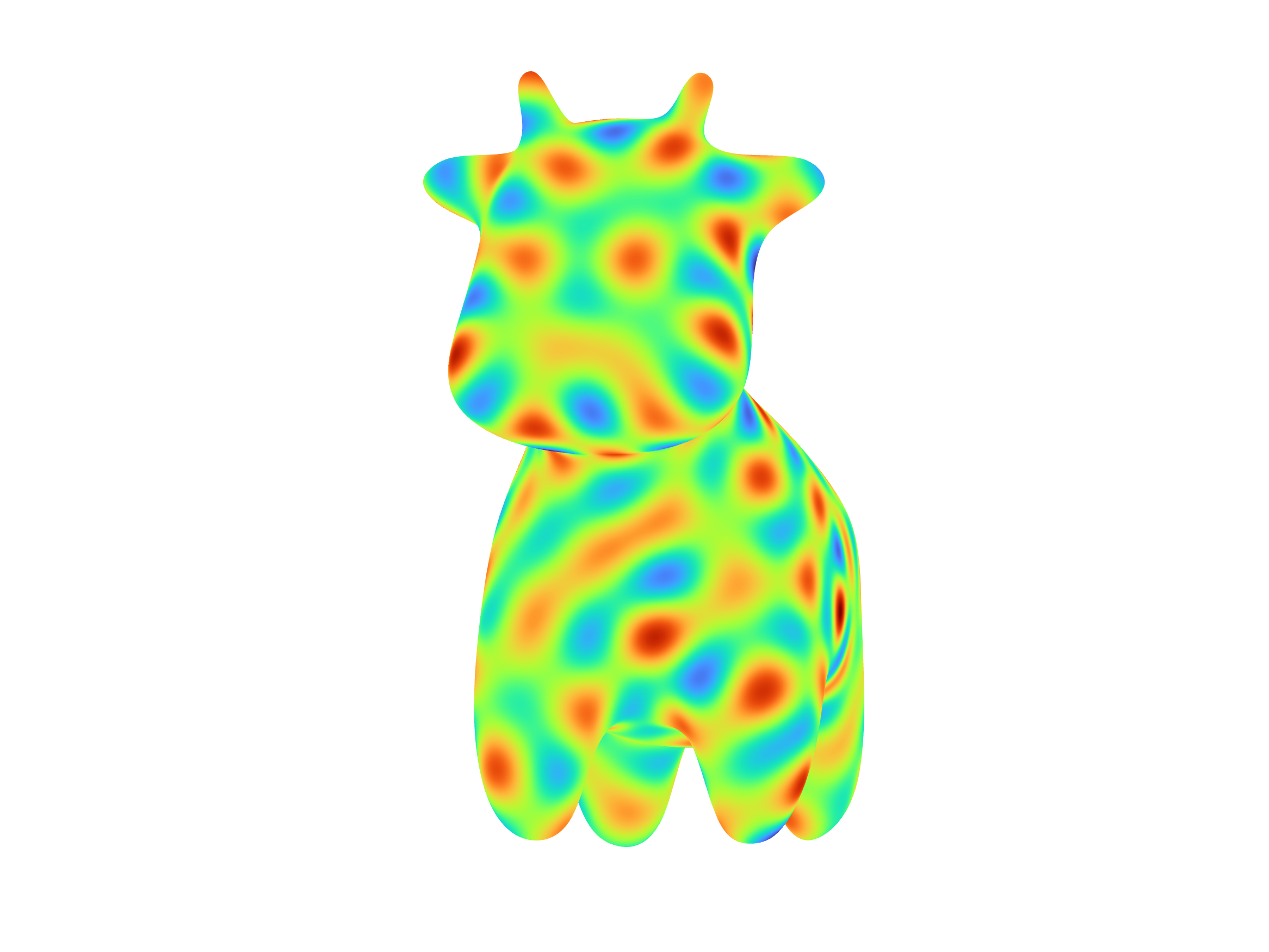}
  \includegraphics[trim={6.5cm, 1.5cm, 6cm, 1.1cm}, clip, width=\wdth]{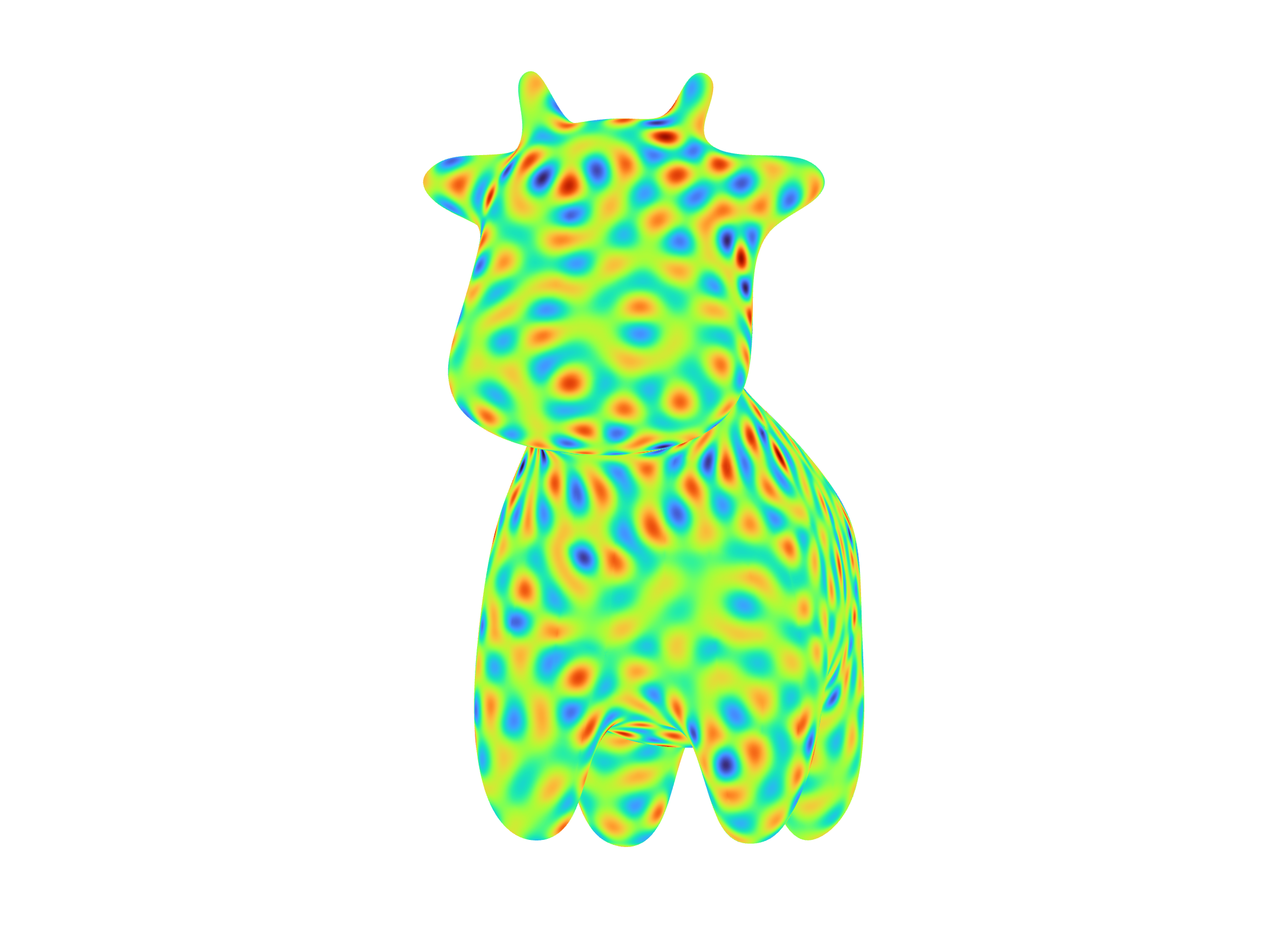}
  \caption{Eigenfunctions $\phi_k$ of the Laplace-Beltrami operator on a cow
  manifold for \mbox{$k=2,4,8,78,340,1350$}.}
  \label{fig:eigenfunctions}
\end{figure}
  
The MHT presents a number of computational challenges when compared to the FFT.
First, unless the geometry is extremely simple, the requisite eigenfunctions
$\phi_k$ are not available in closed form. In addition, even when the
eigenfunctions are available in closed form, there is no general fast algorithm
for evaluating linear combinations of these functions. When $\M$ is the sphere
or disk, the eigenfunctions $\phi_k$ are given by spherical harmonics and Bessel
functions, respectively, and one can take advantage of various symmetries and
analytical expansions to develop fast algorithms~\cite{townsend2015fast,
beckman2026nonuniform, slevinsky2019fast, seljebotn2012wavemoth, tygert2010fast,
mohlenkamp1999fast} which rely on the FFT.

However, no such symmetries or explicit expansions are available \textit{a
priori} for the MHT on general manifolds~$\M$. Therefore the eigenfunctions
$\phi_k$ must be computed numerically in general, and no analytical reduction to
complex exponentials can be made. Instead, we must rely on general-purpose
results regarding the eigenfunctions of the Laplace-Beltrami operator and
leverage purely linear-algebraic structure to compress the matrix~$\mtx{\Phi}$.
A suitable compressed linear-algebraic representation of oscillatory operators
is given by the \textit{butterfly factorization}~\cite{oneil2010algorithm,
li2015butterfly, li2017interpolative}, which we apply and analyze here in the
MHT context. The butterfly factorization is a generalization of the hierarchical
algebraic structure of the FFT algorithm to more general operators satisfying a
\textit{complementary low-rank property} in which a particular hierarchy of
matrix blocks are approximately low-rank. Butterfly factorizations have been
successfully applied to compress Fourier integral
operators~\cite{candes2009fast, li2015butterfly, li2017interpolative}, Hankel
transforms~\cite{li2015butterfly,oneil2010algorithm}, and spherical harmonic
transforms~\cite{tygert2010fast,seljebotn2012wavemoth}, among others.

The MHT and related computations appear in applications to computer graphics and
meshing~\cite{vallet2008spectral, patane2018laplacian, levy2006laplace,
reuter2009discrete, seo2010heat}, surface partial differential equations
(PDEs)~\cite{dziuk2013finite, bonito2013afem}, model order
reduction~\cite{alsnayyan2022manifold, alsnayyan2022iso},
statistics~\cite{borovitskiy2020matern, lang2023galerkin}, as well as data
assimilation and kernel learning~\cite{giannakis2019data, belkin2003laplacian}.
A number of compressed representations for the lowest Laplace-Beltrami
eigenfunctions have been proposed in the
literature~\cite{bennighof2004automated,coifman2006diffusion,nasikun2018fast}.
However, these works focus mainly on efficiently computing these eigenfunctions,
and do not explicitly investigate the possibility of a fast transform for
rapidly evaluating linear combinations of many
eigenfunctions~$\{\phi_k\}_{k=1}^m$ at a large number of target
points~$\{\bm{x}_j\}_{j=1}^n \subset \M$ (or of a fast adjoint transform which
maps function values to expansion coefficients).

In this work, we introduce a discretization-agnostic butterfly-accelerated
manifold harmonic transform (BF-MHT) and demonstrate that it can significantly
reduce the cost to store and apply the MHT. We begin by describing the butterfly
factorization and the trees necessary to compress the MHT matrix. Then, in the
case of the two-dimensional periodic Fourier transform, we prove asymptotic
complexity results regarding the cost to store and apply a version of the
BF-MHT. Finally, we provide numerical experiments which verify our complexity
theory and illustrate the performance of the BF-MHT in a number of applications
in computer graphics, statistics, and machine learning.

\paragraph{Notation}

Throughout this work, we use Golub and Van Loan notation~\cite{golub1983matrix},
i.e.~\texttt{MATLAB}-style submatrix notation. For example, given a matrix
$\mtx{\Phi} \in \C^{n \times m}$ and index sets $\tau := [j_1,j_2,\dots,j_p]$
and $\nu := [k_1,k_2,\dots,k_q]$, we will denote the corresponding $p \times q$
submatrix by
\begin{align}
  \mtx{\Phi}(\tau, \nu) = \mat{
    \mtx{\Phi}(j_1,k_1) & \cdots & \mtx{\Phi}(j_1,k_q) \\
    \vdots & \ddots & \vdots \\
    \mtx{\Phi}(j_p,k_1) & \cdots & \mtx{\Phi}(j_p,k_q) \\
  }.
\end{align}
We denote column and row subsets by
\begin{equation}
  \mtx{\Phi}(:, \nu) = \mtx{\Phi}([1,\dots,n], \nu) \quad \text{and} \quad
  \mtx{\Phi}(\tau, :) = \mtx{\Phi}(\tau,[1,\dots,m])
\end{equation}
respectively. We use $\abs{\tau} = p$ to denote the size (cardinality) of an
index set.

\section{The butterfly factorization}

Discretizing oscillatory operators such as the Fourier transform generally
results in full-rank matrix representations $\mtx{\Phi}$. Nevertheless, many
such matrices exhibit a \textit{complementary low-rank property} --- a property
by which all contiguous submatrices~$\mtx{\Phi}(\tau, \nu)$ with the same number
of entries~$\abs{\tau}\cdot\abs{\nu}$ have constant or nearly-constant ranks.
Butterfly algorithms begin by constructing a tree on the row indices of the
matrix, and another tree on the column indices. One can then move down the row
tree (halving the number of rows) while simultaneously moving up the column tree
(doubling the number of columns). In doing so, the number of entries in each
block is held constant, and the complementary low-rank property can be exploited
in a hierarchical fashion at each level. Figure~\ref{fig:BF-levels} provides a
graphical representation of the row and column trees and the corresponding
low-rank matrix blocks at each level. Various styles of factorization have been
presented in the literature, including column-wise, row-wise, or hybrid type
factorizations~\cite{li2017interpolative, li2015butterfly, candes2009fast,
pang2020interpolative, zheng2023efficient, michielssen1996multilevel}. In the
nomenclature of~\cite{liu2021butterfly}, we utilize a \textit{row-wise}
butterfly factorization in this work.

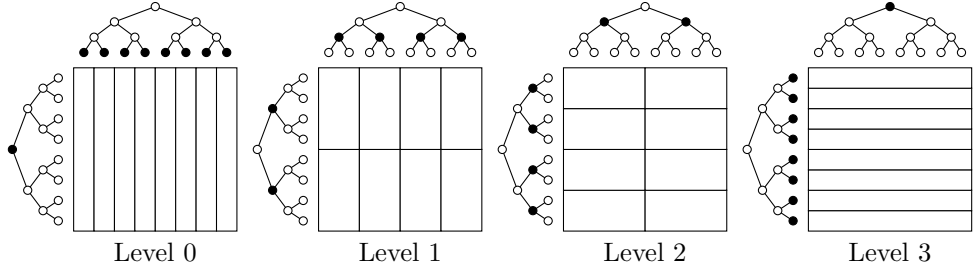
\begin{figure}
  \centering
  \input{./figures/BF-levels-tikz.tex}
  \caption{Row tree, column tree, and resulting low-rank blocks of $\mtx{\Phi}$
  at levels $\ell = 0, \dots, 3$.}
  \label{fig:BF-levels}
\end{figure}

We now describe a particular form of the butterfly factorization algorithm,
presented originally in~\cite{oneil2010algorithm}. To align with the Fourier
transform and MHT context of Section~\ref{sec:intro}, we refer to the first
dimension (the row index) as representing space, and the second dimension (the
column index) as representing eigenvalue or frequency. This corresponds to a
``synthesis'' calculation which evaluates linear combinations of basis functions
as in~\eqref{eq:2D-DFT} and~\eqref{eq:MHT-sum}. We assume binary trees $T_x$ and
$T_\lambda$ have been constructed whose nodes are index sets on the rows and
columns of $\mtx{\Phi}$ respectively. The precise choice of trees may affect the
level of resulting compression, but does not alter the algorithm, and this
construction generalizes in a straightforward way to generic $k$-ary trees, and
to two or three dimensions using quadtrees or octrees.

To begin, at level $\ell = 0$ (see Figure~\ref{fig:BF-levels}) we compute a
low-rank factorization of the block column
\begin{align}
  \mtx{\Phi}(:,\nu) = \mtx{U}^{}_{:\nu} \mtx{V}_{:\nu}^*
\end{align} 
for each leaf node $\nu$ in the frequency tree $T_\lambda$. We note that the
butterfly algorithm is agnostic to the particular type of low-rank factorization
used here and throughout, and many are suitable, e.g. SVD, QR, interpolative
decomposition, etc. 

At levels $\ell = 1, \dots, L$ we use the low-rank factors from the previous
level to compute a low-rank factorization of each complementary block at the
current level. Let $\tau$ be a node at level $L-\ell$ of the space tree $T_x$
with parent $p$, and let $\nu$ be a at level $\ell$ of the frequency tree
$T_\lambda$ with children $c_1$ and $c_2$. Figure~\ref{fig:local-labels}
illustrates the relevant subtrees and their node labels. Assume we have computed
low-rank factorizations of $\mtx{\Phi}(p, c_1)$ and $\mtx{\Phi}(p,c_2)$. We can
then move down the space tree and up the frequency tree one level and write
$\mtx{\Phi}(\tau,\nu)$ in terms of the previously computed factors
\begin{align}
  \mtx{\Phi}(\tau,\nu)
  &= \mat{\mtx{\Phi}(p,c_1) & \mtx{\Phi}(p,c_2)}(\tau, :) \\
  &= \mat{\mtx{U}^{}_{pc_1} \mtx{V}_{pc_1}^* & \mtx{U}^{}_{pc_2} \mtx{V}_{pc_2}^*}(\tau, :) \\
  &= \mat{\mtx{U}_{pc_1}(\tau, :) & \mtx{U}_{pc_2}(\tau, :) } \mat{\mtx{V}_{pc_1}^* & \\ & \mtx{V}_{pc_2}^*}.
\end{align}
Since, by assumption, the matrix admits the complementary low-rank property, we
can then compute a low-rank factorization of the first matrix above:
\begin{align}
  \mat{\mtx{U}_{pc_1}(\tau, :) & \mtx{U}_{pc_2}(\tau, :)}
  = \mtx{U}_{\tau\nu}^{} \mtx{R}_{\tau\nu}^*.
\end{align}
This yields a low-rank factorization of the desired matrix block
\begin{align}
  \mtx{\Phi}(\tau, \nu) 
  = \mtx{U}_{\tau\nu}^{} \mtx{R}_{\tau\nu}^* \mat{\mtx{V}_{pc_1}^* & \\ & \mtx{V}_{pc_2}^*}
  = \mtx{U}^{}_{\tau\nu} \mtx{V}_{\tau\nu}^*,
\end{align}
where $\mtx{R}_{\tau\nu}$ is the \textit{transfer matrix} which constructs the
row basis $\mtx{V}_{\tau\nu}$ for the current level implicitly using the row
bases $\mtx{V}_{pc_1}, \mtx{V}_{pc_2}$ from the previous level. An analogous
process is carried out for the remaining blocks whose parent is $p$ in the space
tree, after which we can discard the column bases $\mtx{U}_{pc_1}$ and
$\mtx{U}_{pc_2}$ from the last level. We repeat this computation for each block
at each level $\ell = 1,\dots,L$ in an \textit{upward pass} through the
frequency tree. Finally, we obtain a compressed representation of $\mtx{\Phi}$
consisting of row bases $\mtx{V}_{:\nu}$ for each leaf node $\nu$ in
$T_\lambda$, transfer matrices $\mtx{R}_{\tau\nu}$ for each block at each level
$\ell=1,\dots,L$, and column bases $\mtx{U}_{\tau:}$ for each leaf node $\tau$
in $T_x$. This factorization can be applied to a vector by first multiplying
with the column basis matrices, then multiplying with each transfer matrix at
each level $\ell = L,L-1,\dots,0$ in a \textit{downward pass} through the
frequency tree. See~\cite{oneil2010algorithm} for a detailed description of the
row-wise butterfly factorization and matrix-vector product algorithms.

\begin{figure}
  \hspace*{0.08\textwidth}
  \begin{tikzpicture}[ndsty/.style={circle,draw,minimum width=7mm}]
    \node[yshift=8mm,rectangle,draw,dotted,thick,minimum width=20mm] {Space};
    \node[yshift=8mm, xshift=50mm,rectangle,draw,dotted,thick,minimum width=20mm]
    {Frequency}; \node[ndsty] (gamma) {$p$} child {node[ndsty,yshift=2mm]
    {$\tau$}} child {node[ndsty,yshift=2mm] {}};
    \node[ndsty, xshift=50mm] {$\nu$}
      child {node[ndsty,yshift=2mm] (sigma) {$c_1$}} child
      {node[ndsty,yshift=2mm] {$c_2$}}; \node[yshift=0mm, xshift=38mm] {Level
      $\ell$}; \node[yshift=-13mm, xshift=-22mm] {Level $L-\ell$};
  \end{tikzpicture}
  \caption{Local node labels for relevant subtrees in the space and frequency trees $T_x$ and $T_\lambda$.}
  \label{fig:local-labels}
\end{figure}
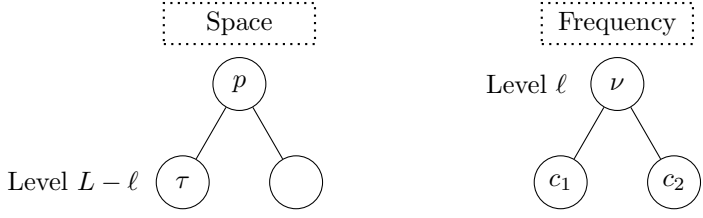

Let $r_\epsilon(\tau, \nu)$ denote the $\epsilon$-rank of the block
$\mtx{\Phi}(\tau, \nu)$. Then each transfer matrix $\mtx{R}_{\tau\nu}$ is of
size $\big( r_\epsilon(p,c_1) + r_\epsilon(p,c_2) \big) \times r_\epsilon(\tau,
\nu)$, so that if all complementary blocks are of exactly $\epsilon$-rank~$r$,
then all transfer matrices are of size $2r \times r$. For a matrix $\mtx{\Phi}
\in \C^{n \times m}$ with $m \leq n$, we store $\bO(m)$ transfer matrices at
each of the $\bO(\log m)$ levels, along with column basis matrices
$\mtx{U}_{\tau:}$ at level $L$. This yields an $\bO\big(rn + r^2 m\log m\big)$
storage and application cost. However, depending on the particular rank
structure of $\mtx{\Phi}$, the ranks of complementary blocks may fluctuate or
grow with~$n$ or~$m$, and as a result, the storage and application cost may
exceed the quasilinear cost of the fixed-rank case. This is shown to be true
theoretically and empirically for the MHT in Sections~\ref{sec:rank-analysis}
and~\ref{sec:numerical-experiments} respectively.

We turn now to the cost of the necessary precomputations for the BF-MHT. The
butterfly factorization cost is dominated by the first level, where we must
compute low-rank factorizations of $\bO(m)$ matrices, each of size $n \times
\bO(1)$. This leads to an $\bO(r^2nm)$ factorization cost. A number of
approaches have been proposed to reduce the cost of this precomputation,
including linear cost interpolative decompositions~\cite{pang2020interpolative,
engquist2009fast, chiu2013sublinear, cortinovis2025sublinear} and residual phase
functions~\cite{demanet2012butterfly, candes2009fast}. Unfortunately, these
approaches rely on difficult-to-verify incoherence properties of the submatrices
of $\mtx{\Phi}$, require the physical locations $\{\bm{x}_j\}_{j=1}^n$ to lie on
specific tensor product grids, or necessitate knowledge of a phase function
$\psi$ for which $K(\bm{x}, \bm{\omega}) \approx e^{i\psi(\bm{\omega},
\bm{x})}$. Thus none of these techniques are directly applicable to the BF-MHT.
Furthermore, computing $m$ Laplace-Beltrami eigenfunctions generally requires at
least $\bO(nm)$ operations, and is thus the bottleneck relative to the
factorization step in practice. Therefore, we do not make any serious attempt
here to accelerate the butterfly factorization step in the MHT. When compressing
to relatively low accuracies --- for example 3 digits of relative precision ---
we find that random subsampling methods for computing interpolative
decompositions~\cite{pang2020interpolative} are generally sufficient, but their
effectiveness is generally heuristic.

\begin{algorithm2e}[t]
  \caption{Butterfly factorization} \label{alg:standard-BF}
  \include{./figures/standard-BF-algo.tex}
\end{algorithm2e}

\begin{remark}
  By arranging the transfer matrices $\mtx{R}_{\tau\nu}$ at each level into a
  matrix, the butterfly factorization can be equivalently written as a product
  of $\bO(\log m)$ sparse matrices each with $\bO(r^2 m)$ entries. This can be a
  useful viewpoint for various computational tasks. However, we find the
  sparsity structure of these factor matrices to be less intuitive for our
  purposes, and thus focus on the tree-based construction given above.
\end{remark}

\section{Fiedler trees}

In order to compute a butterfly factorization of the MHT, we must first
construct a tree on the row indices of $\mtx{\Phi}$ which represent points in
the spatial domain $\M$, and a tree on the column indices of $\mtx{\Phi}$ which
represent eigenvalues in the interval $[0, \lambda_m]$. Forming a binary tree on
the interval $[0, \lambda_m]$ is straightforward. However, building a suitable
tree on the manifold $\M$ is more subtle. In some simple cases like the sphere,
torus, or deformations thereof, one can use an underlying two-dimensional
parameterization of the surface to build a tree whose elements at each level
have approximately equal areas. 

For general surfaces, one straightforward approach is to subdivide $\M$
according to an octree in the ambient space. For simple convex domains this
approach may be sufficient in practice. However, it often results in octree
boxes containing submanifolds of $\M$ which are close in the extrinsic Euclidean
metric but may be arbitrarily far apart in geodesic distance. As a result,
eigenfunctions $\phi_k$ of similar orders $k$ within a given octree box may have
highly dissimilar structures, reducing compressibility and causing the block
ranks within the butterfly factorization to grow unnecessarily quickly.

To avoid such issues, we instead use a \textit{Fiedler
tree}~\cite{szlam2005diffusion, berger2010fiedler}, which respects the intrinsic
geometry of $\M$ relevant to the MHT. The construction of a Fiedler tree takes
advantage of two essential observations. First, by the Courant Nodal Domain
Theorem~\cite{chavel1984eigenvalues}, the second eigenfunction~$\phi_2$ of the
Laplace-Beltrami operator divides $\M$ into exactly two nodal domains --- one on
which $\phi_2 \leq 0$ and one on which $\phi_2 > 0$. If $\M$ is a manifold with
boundary, then this fact holds for any suitable choice of boundary conditions,
e.g. Dirichlet, Neumann, or mixed. In addition, partitioning a manifold
according to the value of $\phi_2$ provides an approximation to the cut of~$\M$
into submanifolds which minimizes the ratio of the length of the cut to the
volume of each of the resulting submanifolds~\cite{cheeger1970lower,
fiedler1975property, shi2000normalized}. In other words, $\phi_2$ divides $\M$
into two pieces with approximately equal surface areas and little connection
between the pieces. 

The Fiedler tree applies these observations recursively --- we compute $\phi_2$
on the original manifold, divide it into two submanifolds according to the sign
of $\phi_2$, compute a new eigenfunction $\phi_2$ on each of the submanifolds,
divide each one into two new submanfiolds using the sign of their respective
eigenfunctions $\phi_2$, and so on. Repeating this partitioning process until
$\bO(1)$ points $\bm{x}_j$ lie in each submanifold yields a Fiedler tree on
$\M$.

\begin{figure}[t]
  \centering
  \includegraphics[width=0.8\textwidth]{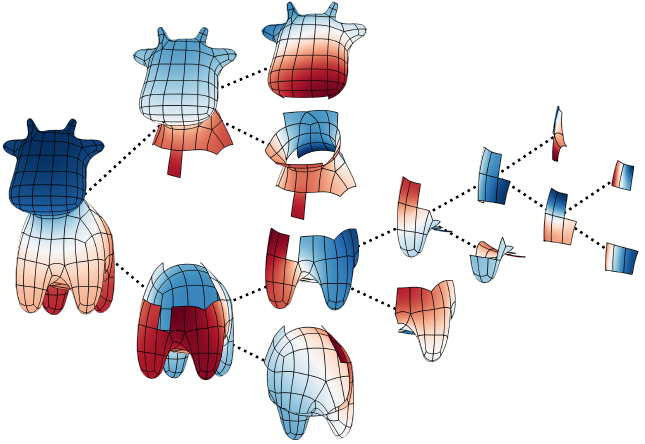}
  \caption{A subset of the nodes of a cut-cell-free Neumann Fiedler tree on the
  cow mesh, with the second Neumann eigenfunction $\phi_2$ plotted on each
  submanifold.}
  \label{fig:fiedler}
\end{figure}

\begin{algorithm2e}[t]
  \caption{Construct Fiedler tree} \label{alg:fiedler}
  \include{./figures/fiedler-algo.tex}
\end{algorithm2e}

Several technical choices must now be made. First, we must choose a boundary
condition to enforce in the eigenproblem on each submanifold. For this, we use a
homogeneous Neumann condition and solve
\begin{align} \label{eq:neumann-eigenproblem}
  \left\{
  \begin{aligned}
    \LBO \phi &= \lambda \phi, & \quad &\text{in } \M, \\
    \pder{\phi}{n} &= 0 & \quad &\text{on } \partial\M.
  \end{aligned}
  \right.
\end{align}
Other choices of boundary condition are possible and will result in slightly
different trees. Next, we must determine how to partition $\M$ given $\phi_2$.
In~\cite{berger2010fiedler}, the authors use a linear approximation to the level
set $\phi_2 = 0$ on each triangular patch $P$, introducing new cut cells into a
linear finite element discretization of $\M$. However, because we require only a
rough partitioning of $\M$ and wish to remain agnostic to the surface
discretization method, we forgo cut cells and instead subdivide $\M$ according
to the sign of the integral $\int_P \phi_2(\bm{x}) \dif{\bm{x}}$ over each
element $P$. This approach is simple to implement, and can be used regardless of
whether the elements $P$ are low- or high-order, triangular or quadrilateral. If
$\M$ is instead defined in an implicit or \textit{mesh-free} way as the
collection of points $\{\bm{x}_j\}_{j=1}^n$, and the eigenfunctions are computed
using e.g. a kernel-based approach, then one can simply partition $\M$ by the
sign of $\phi_2(\bm{x}_j)$ at each point. Algorithm~\ref{alg:fiedler} provides
pseudocode for the recursive construction of a Fiedler tree, and
Figure~\ref{fig:fiedler} illustrates the partitioning process on an example
high-order quadrilateral mesh. In this figure, we plot the
eigenfunctions~$\phi_2$ used on various submanifolds at each level.

\section{The butterfly-accelerated manifold harmonic transform}

Having developed a tree structure on $\M$, we now have all the prerequisite
tools to apply the butterfly factorization to the MHT. However, we note that
applying Algorithm~\ref{alg:standard-BF} directly requires that we first compute
column bases $\mtx{U}_{:\nu}$ for every leaf node $\nu$ in $T_\lambda$. At this
point in the algorithm, the necessary working memory is $\bO(nm)$ in general. As
our objective is to compute a compressed representation of $\mtx{\Phi}$ for
cases when the dense matrix may not fit in working memory, this step may be
infeasible.

In order to reduce the memory requirements of the algorithm so that the
necessary memory is bounded at all times by the size of the final compressed
representation up to a small constant factor, we reorganize the factorization
process using a post-order traversal of the nodes in $T_\lambda$ as suggested
in~\cite{tygert2010fast}. The post-order traversal begins at the left-most leaf
node and proceeds in a ``left, right, parent'' fashion. We can thus compute a
small number of Laplace-Beltrami eigenfunctions $\phi_k$ corresponding to a leaf
node in the frequency tree $T_\lambda$, and compress this block of
eigenfunctions as much as possible before computing the next set of columns.
Algorithm~\ref{alg:streaming-BF} and Figure~\ref{fig:streaming-BF} provide
pseudocode and a graphical representation for this streaming BF-MHT algorithm.

\begin{algorithm2e}[p]
  \caption{Streaming butterfly factorization} \label{alg:streaming-BF}
  \include{./figures/streaming-BF-algo.tex}
\end{algorithm2e}

\begin{figure}[p]
  \centering
  \input{./figures/streaming-BF-tikz.tex}
  \caption{A schematic depiction of the streaming butterfly factorization. The
    nodes of the column (frequency) tree are traversed using a post-order
    traversal. When leaf nodes are visited, new bands of eigenvectors are
    streamed. When an internal node is visited, the partial factors
    corresponding to that node's children are merged and split. Eight steps of
    the algorithm are shown in sequence, with the block boxes indicating the
    order of the steps.  Note that the nodes in the row and column trees being
    visited at each step are marked in the corresponding trees.}
    \label{fig:streaming-BF}
\end{figure}
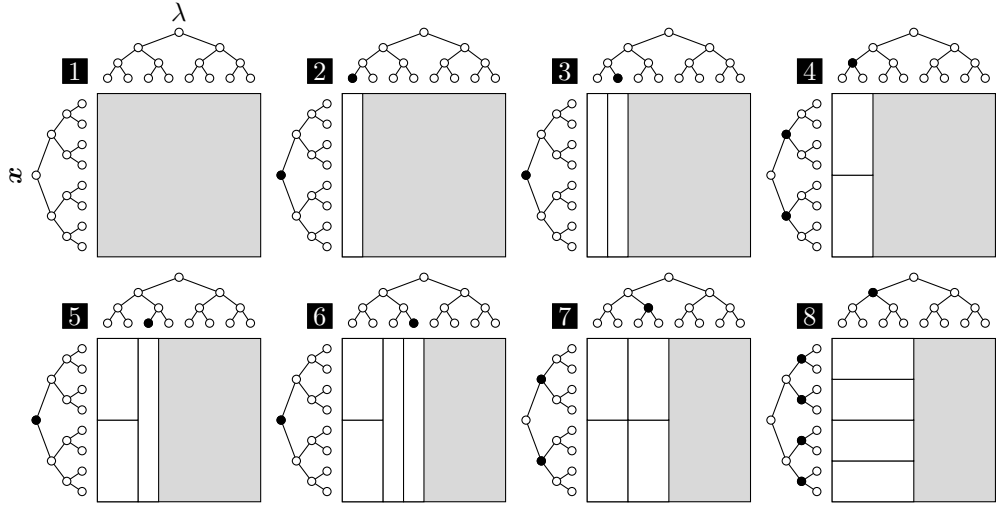

\section{Rank and complexity analysis} \label{sec:rank-analysis}

We now establish asymptotic storage and complexity results for our algorithm to
illuminate various aspects of the BF-MHT's performance which will be observed in
numerical examples and applications in the following sections. Analyzing the
rank of blocks of the manifold harmonic transform in the general case is
challenging because the eigenfunctions and eigenvalues are not known in closed
form. However, as we saw in Section~\ref{sec:intro}, for the flat torus $\M =
[-\pi, \pi]^2$ the manifold harmonic transform is exactly the two-dimensional
Fourier transform with eigenvalues $\lambda_k = k_1^2 + k_2^2$ and
eigenfunctions $\phi_k(\bm{x}) = e^{i(k_1x_1 + k_2x_2)}$ for $k_1,k_2 \in
\mathbb{Z}$. Therefore, we develop here explicit bounds on the $\epsilon$-rank
of blocks of the two-dimensional Fourier transform matrix. We will see that
these results agree with numerical experiments on other manifolds, and we thus
conjecture that they extend to the general case.

In order to move from analysis of the continuous Fourier kernel to numerical
ranks of discrete matrices, we require the following definition of the
$\epsilon$-rank of a kernel function $K$. 

\begin{definition}
  A bivariate function $K : D_x \times D_\omega \to \C$ is said to have
  $\epsilon$-rank $r_\epsilon$ if $r_\epsilon$ is the smallest integer for which
  there exist functions $u_\ell : D_x \to \C$ and $v_\ell : D_\omega \to \C$
  such that
  \begin{align}
    \abs{K(\bm{x}, \bm{\omega}) - \sum_{\ell=1}^{r_\epsilon} u_\ell(\bm{x}) \overline{v_\ell(\bm{\omega})}} \leq \epsilon
  \end{align}
  for all $\bm{x} \in D_x, \bm{\omega} \in D_\omega$.
\end{definition}

This definition implies that for any two sets of points $\{\bm{x}_j\}_{j=1}^n
\subset D_x$ and $\{\bm{\omega}_k\}_{k=1}^m \subset D_\omega$, the kernel matrix
$\mtx{K}_{jk} = K(\bm{x}_j, \bm{\omega}_k) \in \C^{n \times m}$ admits a rank
$r_\epsilon$ factorization such that $\norm[\text{max}]{\mtx{K} -
\mtx{U}\mtx{V}^*} \leq \epsilon$ with $\mtx{U} \in \C^{n \times r_\epsilon}$ and
$\mtx{V} \in \C^{m \times r_\epsilon}$. This choice of norm dramatically
simplifies the analysis that follows, but the results presented here could be
extended to other choices of error norms. 

We continue by introducing a simple bound which will be used extensively in the
proofs that follow.

\begin{lemma} \label{lem:stirling} Take $k \in \Z$ with $k > 0$ and $x \in \R$
  with $x > 0$. Then we have
  \begin{align}
      \abs{J_k(x)}
      \leq \frac{(x/2)^k}{k!}
      \leq e^{\xi} \left(\frac{x}{2\xi}\right)^k,
  \end{align}
  where~$J_k$ is the Bessel function of the first kind of order~$k$. For any
  $\xi > 0$. In particular, taking $\xi = \frac{ex}{2}$ gives
  \begin{equation}
    \abs{J_k(x)} \leq e^{\frac{ex}{2}-k}.
  \end{equation}
\end{lemma}

\begin{proof}
  The first inequality is~\cite[10.14.4]{olver2010nist}, and Stirling's formula
  gives the second.
\end{proof}

Consider the Fourier kernel~$K(\bm{x}, \bm{\omega}) = e^{i\bm{\omega}\cdot
\bm{x}}$ in two dimensions. Given access to the two-dimensional frequencies
$\bm{\omega}$, one could build a butterfly factorization for the two-dimensional
Fourier transform using quadtrees in space and frequency. The following theorem
demonstrates that the rank of the two-dimensional Fourier kernel from
frequencies $\bm{\omega} \in B_b$ to points $\bm{x} \in B_R$ is of size
$r_\epsilon = \bO(b^2 R^2)$, and thus the $\epsilon$-ranks of all compressed
blocks remain bounded by a fixed $r$ at all levels, resulting in an $\bO(r^2
n\log n)$ algorithm.

\begin{theorem} \label{thm:2D-fourier-rank} The $\epsilon$-rank of the Fourier
  kernel $K(\bm{x}, \bm{\omega}) = e^{i\bm{\omega}\cdot \bm{x}}$ for
  $\bm{\omega} \in B_b$ and $\bm{x} \in B_R$ is bounded by 
  \begin{align}
    r_\epsilon \leq \frac{1}{2} \clg{eb R + \log 2 + \log(\epsilon^{-1})}^2.
  \end{align}
\end{theorem}

\begin{proof}
  Writing $\bm{\omega} \mapsto (\rho, \alpha)$ and $\bm{x} \mapsto (r, \theta)$
  in polar coordinates gives
  \begin{align}
      K(\bm{x}, \bm{\omega}) 
      &:= e^{i\rho r\cos(\alpha - \theta)}.
  \end{align}
  Using \cite[Lemma 3.2]{oneil2007transform}, we obtain
  \begin{align} \label{eq:bessel-low-rank}
      K(\bm{x}, \bm{\omega}) 
      & = \sum_{\ell=0}^\infty c_\ell J_\ell(\rho r) \Big( \cos(\ell\alpha) \cos(\ell\theta) + \sin(\ell\alpha) \sin(\ell\theta) \Big)
  \end{align}
  where $c_0 := 1$ and $c_\ell := 2i^\ell$ for all $\ell > 0$. Applying
  \cite[10.23.2]{olver2010nist} then gives the expansion
  \begin{multline}
    \label{eq:fourier-ball-expansion}
    K(\bm{x}, \bm{\omega}) 
    =\\
    \sum_{\ell=0}^\infty c_\ell \Big( \cos(\ell\alpha) \cos(\ell\theta) + \sin(\ell\alpha) \sin(\ell\theta) \Big) \rho^\ell \sum_{k=0}^\infty  \frac{(-1)^k (\rho^2 - 1)^k (\frac{r}{2})^k}{k!} J_{\ell+k}(r),
  \end{multline}
  which is separable into products of terms involving $(r, \theta)$ and $(\rho,
  \alpha)$ respectively. Assume, without loss of generality, that $\rho > 1$.
  Then using the fact that~\mbox{$\abs{J_k(x)} \leq (x/2)^k/k!$},
  we have
  \begin{multline}
    \label{eq:mult-theorem-bound}
    \abs{c_\ell \cos(\ell\alpha) \cos(\ell\theta) \rho^\ell \frac{(-1)^k (\rho^2 - 1)^k (\frac{x}{2})^k}{k!} J_{\ell+k}(r)} \\
    \begin{aligned}
      &\leq \frac{2}{k!} \rho^\ell (\rho^2 - 1)^k \left(\frac{r}{2}\right)^k
      \abs{J_{\ell+k}(r)}\\ 
      &\leq \frac{2}{(\ell+k)! k!} \left(\frac{\rho r}{2}\right)^{\ell + 2k},
    \end{aligned}
  \end{multline}
  with a similar bound for the sine term. Applying Lemma~\ref{lem:stirling} and
  the fact that $\rho \leq b$, we have
  \begin{align}
    \frac{2}{(\ell+k)! k!} \left(\frac{\rho r}{2}\right)^{\ell + 2k}
    \leq 2e^{eb R - \ell - 2k}.
  \end{align}
  Assuring this expression is bounded by $\epsilon$, we obtain
  \begin{align}
    \ell + 2k < eb R + \log 2 + \log(\epsilon^{-1}).
  \end{align}
  Counting the number of integer pairs $(k, \ell)$ satisfying the above, and
  taking into account that each $\ell$ contributes two terms
  in~\eqref{eq:fourier-ball-expansion} --- one with cosines and one with sines
  --- gives the result.
\end{proof}

In contrast, for the general MHT no such two-dimensional frequency $\bm{\omega}$
is well-defined. Instead, we have access only to the linear ordering of
eigenvalues $\lambda_k$ of the Laplace-Beltrami operator. We now consider the
two-dimensional Fourier transform as a special case of the MHT, and study the
rank of the Fourier kernel from eigenvalues with $\sqrt{\lambda} :=
\norm{\omega} \in [a,b]$ to points $\bm{x} \in B_R$. This corresponds to
considering frequencies in the annulus $\bm{\omega} \in A_a^b := \{\bm{\omega} \
: \ a \leq \norm{\bm{\omega}} \leq b\}$.

One-dimensional intuition and the result of Theorem~\ref{thm:2D-fourier-rank}
that the $\epsilon$-rank of the Fourier kernel from $\bm{\omega} \in B_b$ to
$\bm{x} \in B_R$ scales like $r_\epsilon = \bO\big(\vol(B_b) \vol(B_R)\big) =
\bO(b^2 R^2)$ both suggest that for $\bm{\omega} \in A_a^b$, we should have
\begin{equation}
  r_\epsilon = \bO\big(\vol(A_a^b) \vol(B_R)\big) = \bO((b^2 - a^2)R^2).
\end{equation}
In other words, the rank of the Fourier kernel should depend only on the
\textit{phase space volume}, i.e. the product of the volumes of the physical and
frequency domains. However, this turns out to be far from the case, which we
will see is the driving factor in the formal computational complexity of the
MHT.

Consider first the case of $\norm{\bm{\omega}} = b$. Then the Jacobi-Anger
expansion~\cite{olver2010nist} gives the Fourier coefficients of the Fourier
kernel explicitly
\begin{align}
  K(\bm{x}, \bm{\omega})
  &= e^{ibr\cos(\alpha-\theta)}
  = \sum_{k=-\infty}^\infty i^k J_k(br) e^{ik(\alpha - \theta)}.
\end{align}
As $\abs{J_k(bR)}$ is $\bO(1)$ for $k < bR$ and the Fourier bases are orthogonal
in $L^2([0,2\pi])$, one cannot generally hope to obtain an expansion of $K$ with
$\epsilon$-rank less than $bR$ which is valid for $\bm{\omega}$ in any domain
containing the circle of radius $b$. In particular, this implies that the
$\epsilon$-rank of the Fourier kernel between an annulus $\bm{\omega} \in A_a^b$
and a disk $\bm{x} \in B_R$ must grow at least as fast as the outer radius $b$
of the annulus, \textit{even} if the phase space volume is held constant. 

Before providing a bound on the rank of the annulus-to-disk Fourier kernel, we
derive an explicit formula for the Chebyshev expansion coefficients of the
Bessel function $J_k(\rho r)$ for $\rho$ in any interval $[a,b]$ in terms of
sums of products of Bessel functions. This expansion is central to the proof
that follows, and is of independent interest in other applications which require
a local approximation of the Bessel function~$J_k$.

\begin{lemma}
  \label{thm:bessel-chebyshev}
  Take $k \in \Z$ with $k \geq 0$. Then for all $r > 0$ and $a \leq \rho \leq
  b$,
  \begin{equation}
    \label{eq:bessel-chebyshev}
    J_k(\rho r) 
    = \frac{c_{k0}(r)}{2} + \sum_{\ell=1}^\infty c_{k\ell}(r)
    T_\ell\left(\textstyle\frac{2}{b-a}(\rho - a) - 1\right)
  \end{equation}
  where
  \begin{equation}
    c_{k\ell}(r)
    = \sum_{q=0}^\infty \left. \begin{cases}
      c_{k\ell}^{2q}(r) & \ell \ \text{even} \\
      c_{k\ell}^{2q+1}(r) & \ell \ \text{odd}
    \end{cases} \right\},
\end{equation}
and
\begin{multline}
    c_{k\ell}^{p}(r)
    = \\ 2\eta_p J_{\frac{p+\ell}{2}}\left( {\textstyle\frac{(b-a)r}{4}} \right)
    J_{\frac{p-\ell}{2}} \left( {\textstyle\frac{(b-a)r}{4}} \right) \bigg[
    J_{k-p} \left( {\textstyle\frac{(b+a)r}{2}} \right) 
    + (-1)^p J_{k+p} \left( {\textstyle\frac{(b+a)r}{2}} \right) \bigg]
  \end{multline}
  with $\eta_0 := \frac{1}{2}$ and $\eta_p := 1$ for all $p \neq 0$.
\end{lemma}

\begin{proof}
  Let $\rho = \frac{(b-a)}{2}(\cos\tau + 1) + a$. The Chebyshev coefficients are
  then given by 
  \begin{align}
    c_{k\ell}
    &= \frac{2}{\pi} \int_0^\pi J_k\left( {\textstyle\frac{(b-a)r}{2}} \cos\tau + {\textstyle\frac{(b+a)r}{2}} \right) \cos(\ell\tau) \, d\tau \\
    &= \frac{2}{\pi} \int_0^\pi \left[ \sum_{p=-\infty}^\infty J_p\left( {\textstyle\frac{(b-a)r}{2}} \cos\tau \right) J_{k-p}\left( {\textstyle\frac{(b+a)r}{2}} \right) \right] \cos(\ell\tau) \, d\tau \label{eq:cheb-coef-1}
  \end{align}
  where we have applied the Neumann addition
  formula~\cite[10.23.2]{olver2010nist}. Exchanging the order of integration and
  summation, we obtain integrals of the form
  \begin{multline}
    \int_0^\pi J_p\left( {\textstyle\frac{(b-a)r}{2}} \cos\tau \right) \cos(\ell\tau) \, d\tau
    = \\ \begin{cases}
      \pi J_{\frac{p+\ell}{2}} \left( \frac{(b-a)r}{4} \right) J_{\frac{p-\ell}{2}} \left( \frac{(b-a)r}{4} \right) & p+\ell \ \text{is even}\\
      0 & \ \text{otherwise}.
    \end{cases}
  \end{multline}
  The above equality follows by noting that the integrand is symmetric about
  $\pi/2$ when $p+\ell$ is even, in which case we
  apply~\cite[10.22.13]{olver2010nist}. If $p+\ell$ is odd, the integrand is
  antisymmetric about $\pi/2$ and thus the integral is zero. Plugging this in
  to~\eqref{eq:cheb-coef-1} gives the result.
\end{proof}

We can now use this formula for the Chebyshev coefficients of the Bessel
function to bound the rank of the Bessel kernel~$K_k(\rho, r) = J_k(\rho r)$.

\begin{lemma}
  \label{thm:bessel-rank}
  Take $k \in \Z$. The $\epsilon$-rank of $K_k(\rho, r) := J_k(\rho r)$ for $a
  \leq \rho \leq b$ and $0 \leq r \leq R$ is bounded by
  \begin{align}
    r_\epsilon 
    \leq \clg{\frac{2\xi + 2\log 4 + \log(\epsilon^{-2})}{\log\left(\frac{8\xi}{(b-a)R}\right)}}
  \end{align}
  for any $\xi > \frac{(b-a)R}{8}$.
\end{lemma}

\begin{proof}
  Assume $\ell$ is even. Applying Lemma~\ref{thm:bessel-chebyshev}, bounding the
  absolute values of the Bessel functions of orders $(p-\ell)/2$, $k-p$,
  and~$k+p$ by one, and applying Lemma~\ref{lem:stirling} to the remaining term
  yields the bound
  \begin{equation}
    \begin{aligned}
    \abs{c_{k\ell}(r)}
    &\leq 4 \sum_{q=0}^\infty \abs{J_{q+\frac{\ell}{2}}\left( \frac{(b-a)r}{4} \right)} \\
    &\leq 4 \sum_{q=0}^\infty e^\xi \left(\frac{(b-a)r}{8\xi}\right)^{q+\frac{\ell}{2}} \\
    &= \frac{4 e^\xi \left(\frac{(b-a)r}{8\xi}\right)^{\frac{\ell}{2}}}{1 - \frac{(b-a)r}{8\xi}},
  \end{aligned}
  \end{equation}
  which is less than $\epsilon$ for all
  \begin{align}
    \frac{\ell}{2} > \frac{\xi - \log\left( 1 - \frac{(b-a)r}{8\xi} \right) + \log 4 + \log(\epsilon^{-1})}{\log\left( \frac{8\xi}{(b-a)r} \right)}
  \end{align}
  for all $\xi > \frac{(b-a)r}{8}$. 
  Neglecting the second term in the numerator gives the result.
\end{proof}

We are ready to state and prove our desired rank bound on the annulus-to-disk
Fourier transform. In essence, we show that the Fourier kernel mapping
frequencies $\bm{\omega} \in A_a^b$ to points $\bm{x} \in B_R$ has
$\epsilon$-rank $\bO(bR)$ up to log factors in $(b-a)R$ if this quantity is
sufficiently small, i.e. if the annulus is sufficiently narrow.

\begin{theorem}
  \label{thm:MHT-fourier-rank}
  The $\epsilon$-rank of the Fourier kernel $K(\bm{x}, \bm{\omega}) :=
  e^{i\bm{\omega}\cdot \bm{x}}$ for~\mbox{$a \leq \norm{\bm{\omega}} \leq b$}
  and~$\bm{x} \in B_R$ is bounded by
  \begin{align}
    r_\epsilon 
    \leq \clg{\frac{5}{2} + \log(\epsilon^{-4})} \cdot \clg{\frac{e}{2}bR + \log(\epsilon^{-1})}
  \end{align}
  for any $(b-a)R < 1$.
\end{theorem}

\begin{proof}
  We first expand $K$ in a Fourier series in $\alpha - \theta$, giving
  \begin{align}
    K(\bm{x}, \bm{\omega})
    &= e^{i\rho r\cos(\alpha-\theta)}
    = \sum_{k=-\infty}^\infty i^k J_k(\rho r) e^{ik(\alpha - \theta)}.
  \end{align}
  Applying Lemma~\ref{thm:bessel-chebyshev} then yields
  \begin{align}
    K(\bm{x}, \bm{\omega})
    &= \sum_{k=-\infty}^\infty \sum_{\ell=0}^\infty i^k c_{k\ell}(r) T_\ell\left(\textstyle\frac{2}{b-a}(\rho - a) - 1\right) e^{ik(\alpha - \theta)}.
  \end{align} 
  Applying Lemma~\ref{lem:stirling} to $J_k(\rho r)$ gives
  \begin{align}
    \abs{c_{k\ell}(r)} 
    \leq \abs{J_k(\rho r)}
    \leq e^{\frac{e\rho r}{2} - k},
  \end{align}
  and thus $\abs{c_{k\ell}(r)} < \epsilon$ for all 
  \begin{align}
    k > \frac{e\rho r}{2} + \log(\epsilon^{-1}).
  \end{align}
  Combining this with Lemma~\ref{thm:bessel-rank} for $\xi = 1$, using the facts
  that $2 + 2\log 4 < 5$ and $\log 8 < 2$, and counting the number of pairs
  $(k,\ell)$ for which $\abs{c_{k\ell}(r)}$ may exceed $\epsilon$ gives the
  result.
\end{proof}

\begin{remark}
  With some additional calculations, the~$\log(\epsilon^{-4})$ term in the
  pre-factor in Theorem~\ref{thm:MHT-fourier-rank} could likely be reduced
  to~$\log(\epsilon^{-1})$, but the dependency on~$b$ and~$R$ is unchanged.
\end{remark}

We now employ this rank bound for each compressed block of the BF-MHT to analyze
the complexity of our algorithm applied to the flat torus.
Figure~\ref{fig:annulus-to-disk} shows a diagram of the two-dimensional Fourier
transform as a MHT and its relation to the annulus-to-disk Fourier kernel.

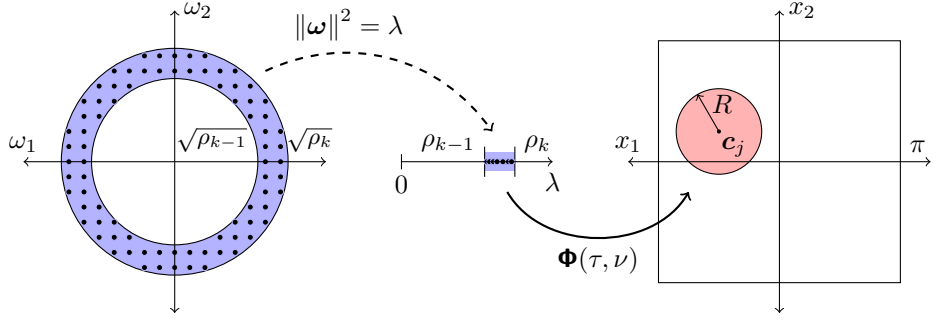
\begin{figure}
  \centering
  \input{./figures/annulus-to-disk-tikz.tex}
  \caption{Diagram of the two-dimensional Fourier transform as a MHT. The
  annulus $A_{\sqrt{\rho_{k-1}}}^{\sqrt{\rho_k}}$ (left) contains
  eigenfrequencies $\bm{\omega} = [k_1, k_2]$ for $k_1, k_2 \in \Z$, denoted by
  black dots. These eigenfrequencies are mapped by $\norm{\bm{\omega}}^2 =
  \lambda$ to the corresponding eigenvalues $\lambda_k$ of
  $\Delta_{[-\pi,\pi]^2}$ (center). Each block $\mtx{\Phi}(\tau,\nu)$ of the MHT
  maps coefficients of eigenfunctions with $\lambda \in [\rho_{k-1}, \rho_k]$ to
  values in a disk $\bm{x} \in B_R(\bm{c}_j)$.}
  \label{fig:annulus-to-disk}
\end{figure}

\begin{theorem} \label{thm:MHT-complexity} Let $\M = [-\pi, \pi]^2$. The BF-MHT
  mapping coefficients to a linear combination of the first $m$ Laplacian
  eigenfunctions at $n$ points is $\bO(n + m^{5/3})$ complexity to store and
  apply.
\end{theorem}

\begin{proof}
  Assume $m \leq n$. In addition, assume that $m = 4^L$ for ease of notation. At
  each level $\ell=0,\dots,L$, for each box index~$k = 1,\dots,4^\ell$ and
  $j=1,\dots,4^{L-\ell}$, we apply a transfer matrix corresponding to
  frequency-space pairs
  \[
    (\bm{x}, \bm{\omega}) \in B_{\sqrt{2}\pi/2^\ell}(\bm{c}_j) \times A_{\sqrt{\rho_{k-1}}}^{\sqrt{\rho_k}},
  \]
  where~$\rho_k := k \lambda_m / 4^{L-\ell}$ is the upper boundary of box $k$ on
  level $\ell$ of a 4-ary tree on~$[0, \lambda_m]$ and~$\bm{c}_j$ is the center
  of box $j$ on level $L-\ell$ of a quadtree on $[-\pi,\pi]^2$.

  Take $a = \sqrt{\rho_{k-1}}$, $b = \sqrt{\rho_k}$, and $R =
  \sqrt{2}\pi/2^\ell$. Combining Theorem~\ref{thm:MHT-fourier-rank} and the fact
  that Weyl's law~\cite{chavel1984eigenvalues} gives $\lambda_m \sim m / \pi$,
  we see that the $\epsilon$-rank of block $(j, k)$ at level $\ell$ is of size
  \begin{align}
    r_\epsilon(\ell, j, k)
    = \bO\bigg(\sqrt{\frac{k \lambda_m}{4^{L-\ell}}} \frac{\sqrt{2}\pi}{2^\ell}\bigg)
    = \bO(\sqrt{k}).
  \end{align}
  At the same time, by Weyl's law the number of eigenvalues in the interval
  $[\rho_{k-1}, \rho_k]$, and therefore an upper bound on the size and rank of
  the corresponding MHT block is
  \begin{equation}
    r_\epsilon(\ell, j, k) 
    = \bO\left( \frac{k \lambda_m}{4^{L-\ell}} - \frac{(k-1) \lambda_m}{4^{L-\ell}} \right)
    = \bO(4^\ell)
  \end{equation}
  Therefore, in early levels of the factorization there are very few eigenvalues
  in each interval $[\rho_{k-1}, \rho_k]$ and thus $r_\epsilon$ is bounded by
  $\bO(4^\ell)$ for each block independent of $k$, despite the fact that the
  continuous rank of the corresponding annulus-to-disk Fourier kernel may be
  much larger. In contrast, in later levels the number of eigenvalues in each
  interval $[\rho_{k-1}, \rho_k]$ is larger than the $\bO(\sqrt{k})$ continuous
  rank of the Fourier kernel, and thus compression occurs.

  Let $p_j$ denote the parent of block $j$ at level $\ell$ in the space tree,
  and let $C_k$ denote the set of children of block $k$ in the frequency tree.
  For levels $\ell = 0, \dots, L-1$, we store and apply transfer matrices which
  map the row bases at level $\ell$ to row bases at level $\ell+1$, resulting in
  a cost at level $\ell$ of
  \begin{align}
    \sum_{j=1}^{4^\ell} \sum_{k=1}^{4^{L-\ell}} \sum_{c_k \in C_k}
    r_\epsilon(\ell, p_j, c_k) \cdot r_\epsilon(\ell+1, j, k)
    = \sum_{j=1}^{4^\ell} \sum_{k=1}^{4^{L-\ell}} r_\epsilon(\ell+1, j, k)^2,
  \end{align}
  as the two numerical ranks in the first line above are within a constant
  factor of each. Even if no compression occurs, moving one level up the
  quadtree can not increase the size of the basis by more than a factor of 16.

  For $\ell$ small, the cost to store and apply the transfer matrices at level
  $\ell$ is
  \begin{align}
    \sum_{j=1}^{4^\ell} \sum_{k=1}^{4^{L-\ell}} \bO(4^{2\ell})
    = \bO(4^{L + 2\ell}).
  \end{align}
  For $\ell$ large, we have instead
  \begin{align}
    \sum_{j=1}^{4^\ell} \sum_{k=1}^{4^{L-\ell}} \bO(k)
    = \bO(4^{2L - \ell}).
  \end{align}
  Equating these two asymptotic complexities gives that level $\ell = \bO(L/3)$
  dominates. As the costs of all other levels decay exponentially in $\ell$ away
  from this level, the overall cost to store and apply row bases at level $\ell
  = 0$ and transfer matrices at levels $\ell=1,\dots,L$ is $\bO(m^{5/3})$.

  All that remains is to determine the complexity of storing and applying column
  bases at the last level, where compressed blocks are mapped to values at
  points in physical space. Let $n_j$ be the number of points in box $j$ on the
  leaf level $L$ of a quadtree on $[-\pi,\pi]^2$. At this level $\ell = L$, we
  store and apply column basis matrices for each block row, yielding a cost of
  \begin{align}
    \sum_{j=1}^{m} \bO\Big(n_j \cdot r_\epsilon(L, j, 1)\Big) 
    &= \sum_{j=1}^{m} \bO(n_j)
    = \bO(n)
  \end{align}
  This gives the result.
\end{proof}

Despite this asymptotic result, we will see in the following section that most
relevant problem sizes lie in the pre-asymptotic regime. In this regime we
observe empirically that the~$\epsilon$-rank of the annulus-to-disk Fourier
kernel~$K : A_a^b \mapsto B_R$ grows roughly like~$(bR)^{2/3}$ rather than
like~$bR$. Using this observed rate gives the following lemma.

\begin{lemma} \label{lem:MHT-complexity-3/2} Let $\M = [-\pi, \pi]^2$. Choose
 $C_1$ and $C_2$ such that the $\epsilon$-rank of the Fourier kernel
 $K(\bm{\omega}, \bm{x}) := e^{i\bm{\omega}\cdot \bm{x}}$ for all $a \leq
 \norm{\bm{\omega}} \leq b \leq C_1$ with $(b-a)R < 1$ and $\bm{x} \in B_R$ is
 bounded by 
  \begin{align}
    r_\epsilon \leq C_2 (bR)^{2/3}.
  \end{align}
  Then the cost to store and apply BF-MHT mapping coefficients to a linear
  combination of the first $m \leq C_1^2/\pi$ Laplacian eigenfunctions at $n$
  points is bounded by $C_3\big(n + m^{3/2}\big)$ for some constant $C_3$
  independent of $a, b,$ and $R$.
\end{lemma}

\begin{proof}
  Weyl's law gives $\lambda_m \sim m/\pi$, so that $m \leq C_1^2/\pi$ guarantees
  that all blocks of the MHT correspond to $\norm{\bm{\omega}} \leq C_1$.
  Following the proof of Theorem~\ref{thm:MHT-complexity} above but using
  $r_\epsilon \sim (bR)^{2/3} \sim k^{1/3}$ in place of $r_\epsilon = \bO(bR) =
  \bO(\sqrt{k})$ implies that the cost to store and apply level $\ell$ of the
  butterfly factorization is
  \begin{align}
    \sum_{j=1}^{4^\ell} \sum_{k=1}^{4^{L-\ell}} C_2 k^{2/3}
    \sim 4^{\frac{5}{3}L - \frac{2}{3}\ell}.
  \end{align}
  Equating this with the $\bO(4^{L + 2\ell})$ complexity from Weyl's law implies
  that level $\ell \sim \frac{L}{4}$ dominates, which yields the result.
\end{proof}

We emphasize that the modified rank growth assumption in the above lemma is
purely empirical. One could write a bound of the form $r_\epsilon \leq C
(bR)^{2/3}$ that is valid up to some fixed maximum value of $bR$ --- this would
indicate the range of $m$ for which one expects~$n + m^{3/2}$ rather than $n +
m^{5/3}$ scaling. However, further investigation is required to determine
whether this~$(bR)^{2/3}$ rate stems from a meaningful structural property of
the Fourier kernel at moderate frequencies, or if it is simply a pre-asymptotic
coincidence.

\section{Numerical experiments} \label{sec:numerical-experiments}

We now provide a number of numerical experiments to demonstrate the scalability
and applicability of our algorithm. We begin by testing the accuracy and
compression properties of the BF-MHT on various manifolds.

\subsection{Two-dimensional discrete Fourier transform}

First, we consider the two-dimensional DFT as a MHT on the flat torus, as we
analyzed in the last section. For various $n$, we butterfly compress the MHT
matrix $\mtx{\Phi}$ whose columns are the first $m = \clg{n/25}$ eigenfunctions
$\phi(x_1, x_2) = e^{i(k_1 x_1 + k_2 x_2)}$ ordered by their corresponding
eigenvalues $\lambda = k_1^2 + k_2^2$. This linear scaling of $m$ with $n$
guarantees that the most oscillatory eigenfunctions are discretized using five
points per wavelength in each dimension to avoid aliasing issues.
Figure~\ref{fig:DFT-2D} shows the size in memory of the butterfly factorization
of $\mtx{\Phi}$ as we increase $n$. Empirically, we observe a clear $\sim
n^{3/2}$ scaling as suggested in Lemma~\ref{lem:MHT-complexity-3/2}, leading to
a factor of ${\sim}33$ reduction in memory for $n = 4^{11} \approx 4.2$M and
$\epsilon = 10^{-3}$. Although the BF-MHT is of course not competitive with the
2D DFT, these results indicate the level of compression that one could expect
for problems of a similar size on general manifolds.

The proof of Theorem~\ref{thm:MHT-complexity} illustrates that the complexity of
the factorization is determined by the $\epsilon$-rank of the annulus-to-disk
Fourier kernel~$K : A_a^b \mapsto B_R$, which scales asymptotically like $bR$ by
Theorem~\ref{thm:MHT-fourier-rank}. Figure~\ref{fig:DFT-2D} displays the upper
bound from Theorem~\ref{thm:MHT-fourier-rank}, as well as the $\epsilon$-rank of
compressed blocks of the MHT matrix $\mtx{\Phi}$ for the largest tested size
of~$n = 10^6$,~$m = 40{,}000$, and $L = 10$ at the dominant level of the
factorization as a function of the corresponding values of $bR$ for each block.
We also plot the $\epsilon$-rank of the circle-to-circle Fourier kernel~$K : C_b
\mapsto C_R$, which is a lower bound on the annulus-to-disk kernel as $C_b
\subset A_a^b$ and $C_R \subset B_R$. We see that when compressing a very large
number of eigenfunctions, e.g. $m > 10^5$, one may begin to see the $r_\epsilon
= \bO(bR)$ scaling which leads to an overall $\bO(n + m^{5/3})$ complexity.
However, for more modest choices of~$m$, the empirical scaling resembles
$r_\epsilon \sim (bR)^{2/3}$. This leads instead to a transform whose overall
storage and apply cost is $\sim n + m^{3/2}$.

Finally, Figure~\ref{fig:DFT-2D} displays the relative $\ell^2$ error
$\|\mtx{\Phi}\vct{v} - \mtx{\Phi}_\epsilon\vct{v}\| / \|\mtx{\Phi}\vct{v}\|$ for
a standard Gaussian random vector $\vct{v}$, where $\mtx{\Phi}_\epsilon$ is the
BF-MHT with tolerance $\epsilon$. In this numerical experiment we fix $n =
160{,}000$, and observe good agreement between the requested tolerance
$\epsilon$ and the resulting relative error in the factorization.

\begin{figure}
  \centering
  \begin{subfigure}[b]{0.28\textwidth}
    \includegraphics[width=\textwidth, trim={0.5cm, 0.5cm, 0.3cm, 0.1cm},
    clip]{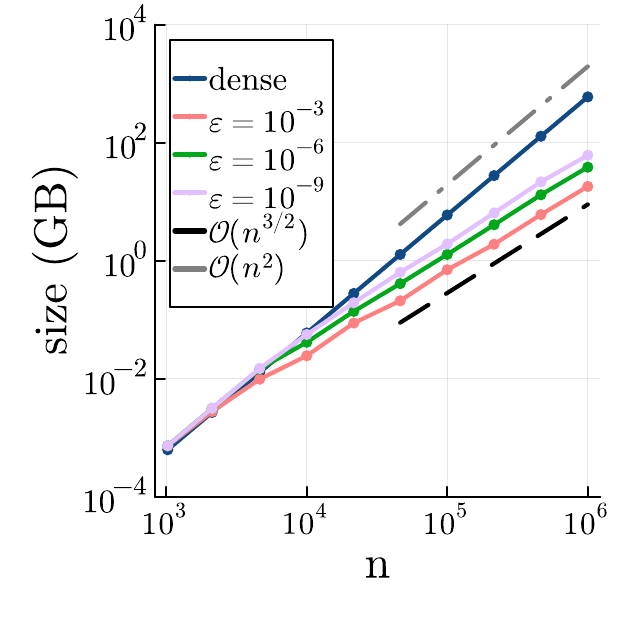}
  \end{subfigure}
  \hfill
  \begin{subfigure}[b]{0.28\textwidth}
    \includegraphics[width=\textwidth, trim={0.5cm, 0.65cm, 0.3cm, 0.1cm}, clip]{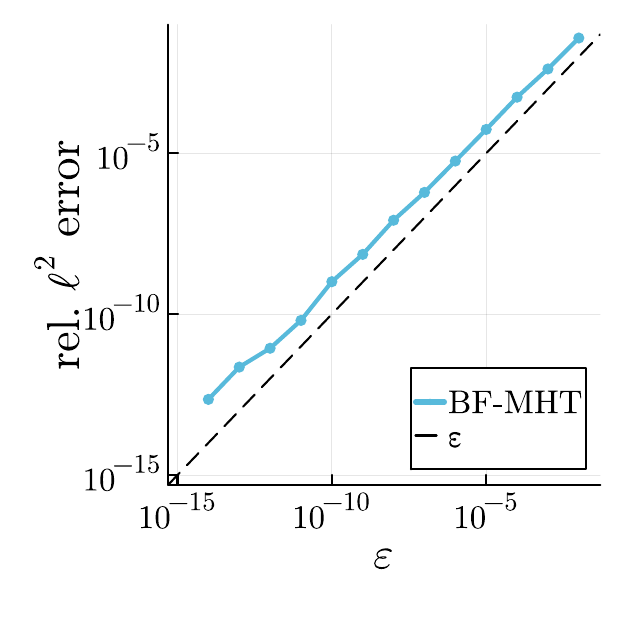}
  \end{subfigure}
  \hfill
  \begin{subfigure}[b]{0.4\textwidth}
    \includegraphics[width=\textwidth, trim={0.0cm, -0.2cm, 0.3cm, 0.1cm}, clip]{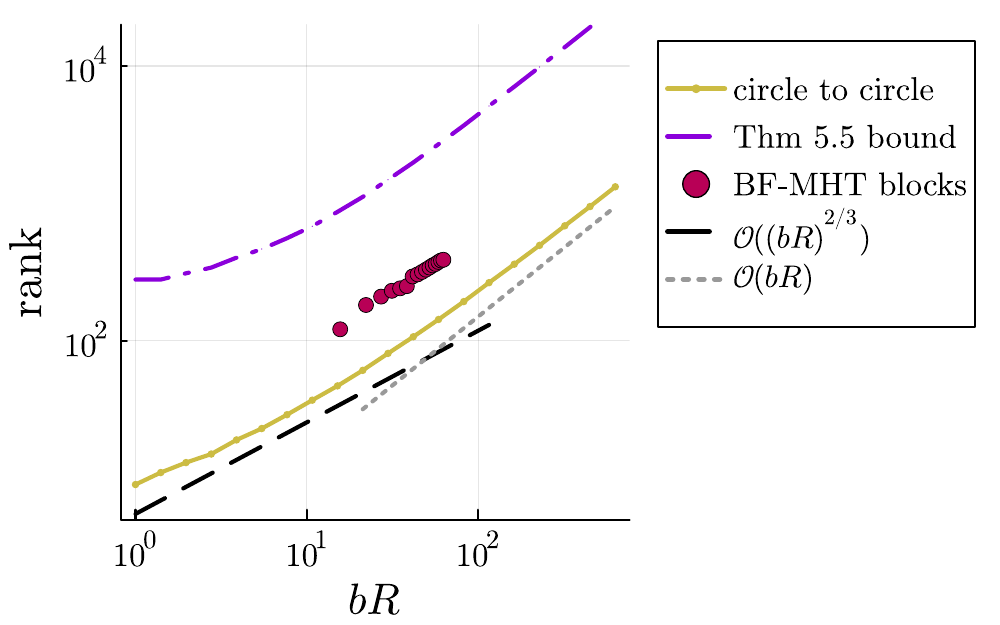}
  \end{subfigure}
  \caption{Numerical experiments for the flat torus $\M = [-\pi,\pi]^2$. Scaling
    with $n$ for $m = \clg{n/25}$ for various tolerances $\epsilon$ (left).
    Matvec relative error as a function of compression tolerance $\epsilon$ for
    $n = 160{,}000$ and $m = 6{,}400$ (center).  $\epsilon$-ranks of
    circle-to-circle Fourier transforms and of compressed blocks of the BF-MHT
    corresponding to annulus-to-disk Fourier transforms with $n = 10^6, m =
    40{,}000,$ and $\epsilon = 10^{-3}$ (right).}
  \label{fig:DFT-2D}
\end{figure}

\subsection{Manifold harmonic transform on a deformed torus} \label{sec:torus}

We now test the BF-MHT on a deformed torus mesh discretized using quadrilateral
patches with $8 \times 8$ Chebyshev nodes on each patch, which we use as our
spatial locations $\{\bm{x}_j\}_{j=1}^n$. As this manifold does not have known
closed-form eigenfunctions, we must compute $\phi_k$ numerically. For this
purpose we use a fast direct solver for the Laplace-Beltrami problem based on a
spectral collocation scheme~\cite{fortunato2024high}. We then employ a Krylov
eigensolver~\cite{stewart2002krylov} to compute $m$ eigenfunctions. For various
$n$ we compute and compress $m = \clg{n/32}$ eigenfunctions using the BF-MHT.
The linear scaling of $m$ with $n$ is natural in this setting, as one must
refine the mesh in proportion to the scale of oscillations in the eigenfunctions
in order to accurately numerically compute them. Figure~\ref{fig:torus}
illustrates an empirical~$\sim n^{3/2}$ scaling for the memory required to store
the BF-MHT. Relative to the $\bO(n^2)$ cost of storing the dense matrix, we see
that the BF-MHT yields a factor of ${\sim}10$ reduction for $n = 51{,}200$.
Recall that the complexity of storing and applying the MHT are asymptotically
equivalent because every stored transfer matrix must be applied to a vector
exactly once, and thus this memory scaling is also a proxy for matvec runtime
scaling. Figure~\ref{fig:torus} also illustrates agreement between the requested
tolerance $\epsilon$ and the relative error in matrix-vector products with the
BF-MHT, just as in the flat torus experiments above.

\begin{figure}[!t]
  \centering
  \begin{subfigure}[b]{0.28\textwidth}
    \includegraphics[width=\textwidth, trim={0.5cm, 0.5cm, 0cm, 0.1cm}, clip]{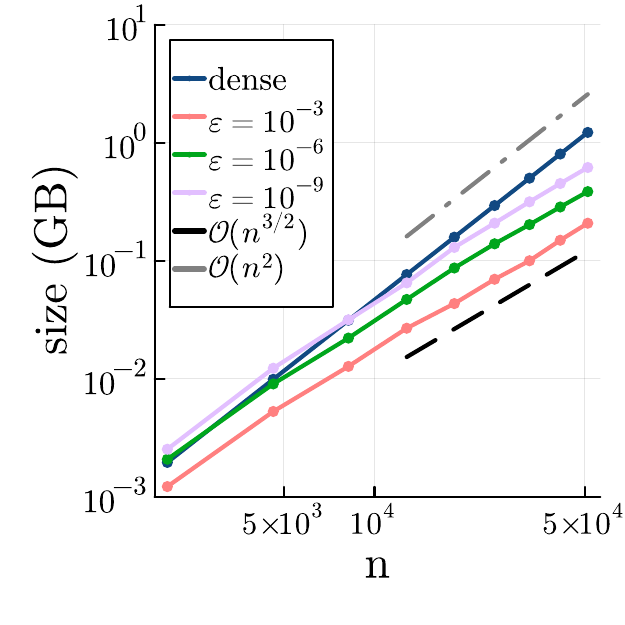}
  \end{subfigure}
  \hfill
  \begin{subfigure}[b]{0.28\textwidth}
    \includegraphics[width=\textwidth, trim={0.5cm, 0.65cm, 0.3cm, 0.1cm}, clip]{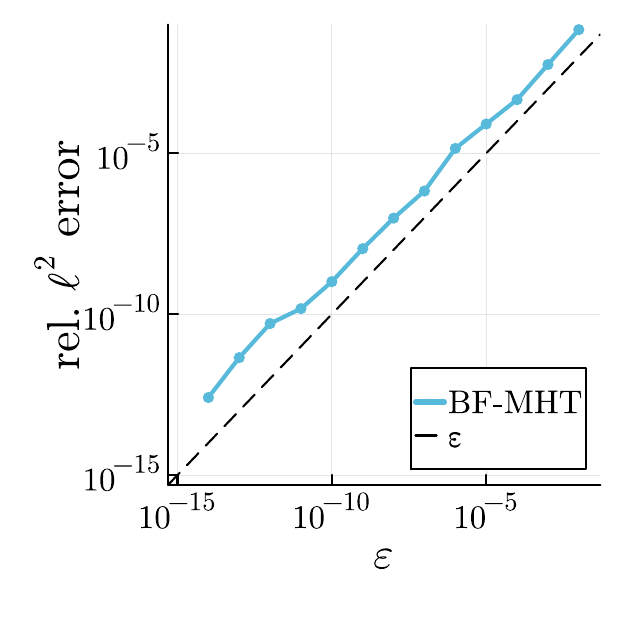}
  \end{subfigure}
  \hfill
  \begin{subfigure}[b]{0.39\textwidth}
    \includegraphics[width=\textwidth, trim={-1cm, 0cm, -1cm, 0cm},
    clip]{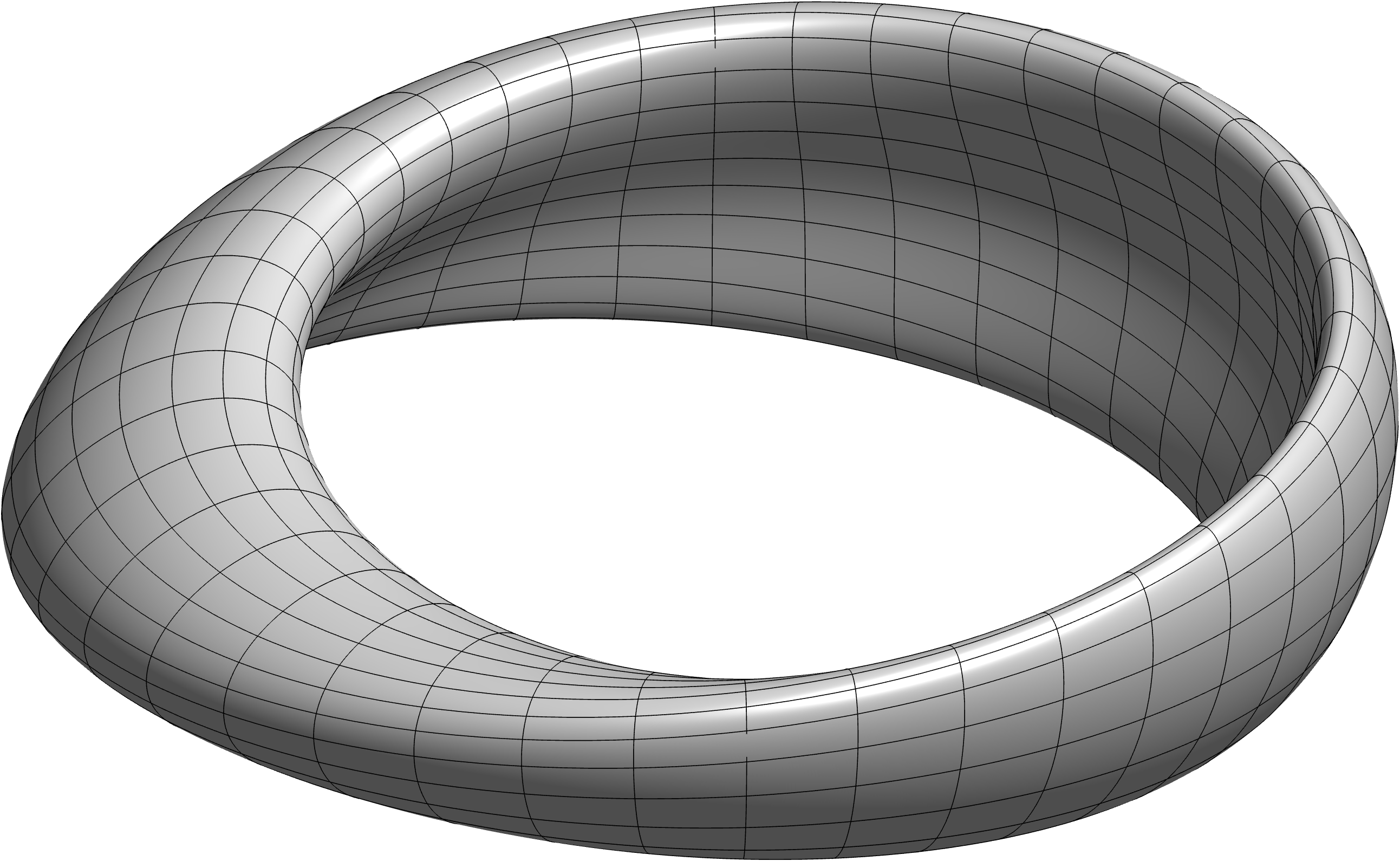} \vspace{0.5cm}
  \end{subfigure}
  \caption{Numerical experiments for a deformed torus. Scaling with $n$ for $m =
  \clg{n/32}$ for various tolerances $\epsilon$ (left). Matvec relative error as
  a function of compression tolerance $\epsilon$ for $n = 51{,}200$ and $m =
  1{,}600$ (center). Mesh of the geometry (right).}
  \label{fig:torus}
\end{figure}

\section{Applications} 

We now demonstrate the performance of the BF-MHT in various applications to
computer graphics, uncertainty quantification, and machine learning. In these
examples, we compress eigenfunctions computed using various discretizations of
the Laplace-Beltrami operator and of the underlying manifold, including finite
elements on triangularizations, spectral collocation on quadrilateral meshes,
and radial basis functions on point clouds. In doing so, we emphasize that the
BF-MHT is discretization agnostic, and applicable across a wide variety of
geometric and computational frameworks.

\subsection{Spectral geometry processing}

The manifold harmonic transform was proposed in~\cite{vallet2008spectral} as a
method for filtering and modifying meshes in computer graphics. In this setting,
one considers the coordinates of the mesh vertices $\bm{x}_1, \dots, \bm{x}_n
\in \R^3$ as point values of three distinct univariate functions, and computes
their coefficients in the basis consisting of the first $m$ Laplace-Beltrami
eigenfunctions on the mesh. This amounts to computing an \textit{inverse} MHT by
solving the least squares problem
\begin{align} \label{eq:inverse-MHT}
  \vct{c}^{(i)} 
  := \argmin_{\vct{c} \in \R^m} \norm[2]{\mtx{\Phi} \vct{c} - \mat{(\bm{x}_1)_i \\ \vdots \\ (\bm{x}_n)_i}}^2
\end{align}
for $i=1,2,3$. The coefficient vectors $\vct{c}^{(i)}$ are then multiplied
pointwise by a filter function $\vct{c}^{(i)}_k \mapsto F(\lambda_k)
\vct{c}^{(i)}_k$ which amplifies or attenuates geometric features on various
lengthscales. Finally, we compute MHTs of the modified coefficient vectors for
each coordinate function, mapping from spectral space back to a modified mesh in
physical space
\begin{align}
  \mat{(\tilde{\bm{x}}_1)_i \\ \vdots \\ (\tilde{\bm{x}}_n)_i}
  = \mtx{\Phi} \mat{F(\lambda_1)\vct{c}^{(i)}_1 \\ \vdots \\ F(\lambda_m) \vct{c}^{(i)}_m}.
\end{align}
After sufficiently many eigenfunctions have been computed to accurately
represent the coordinate functions of the original manifold $\M$, the dominant
computational cost in this pipeline is the computation of the coefficient
vectors using the inverse MHT~\eqref{eq:inverse-MHT}. To accelerate this
computation, we use the BF-MHT within the LSQR iterative least squares
algorithm~\cite{paige1982lsqr}. After the coefficient vectors have been
computed, a modified geometry can be very rapidly computed for a new filter
function $F$ using a single BF-MHT, allowing for fast, interactive mesh design.

\begin{figure}[t]
  \centering
  \begin{subfigure}[T]{0.31\textwidth}
    \includegraphics[width=\textwidth]{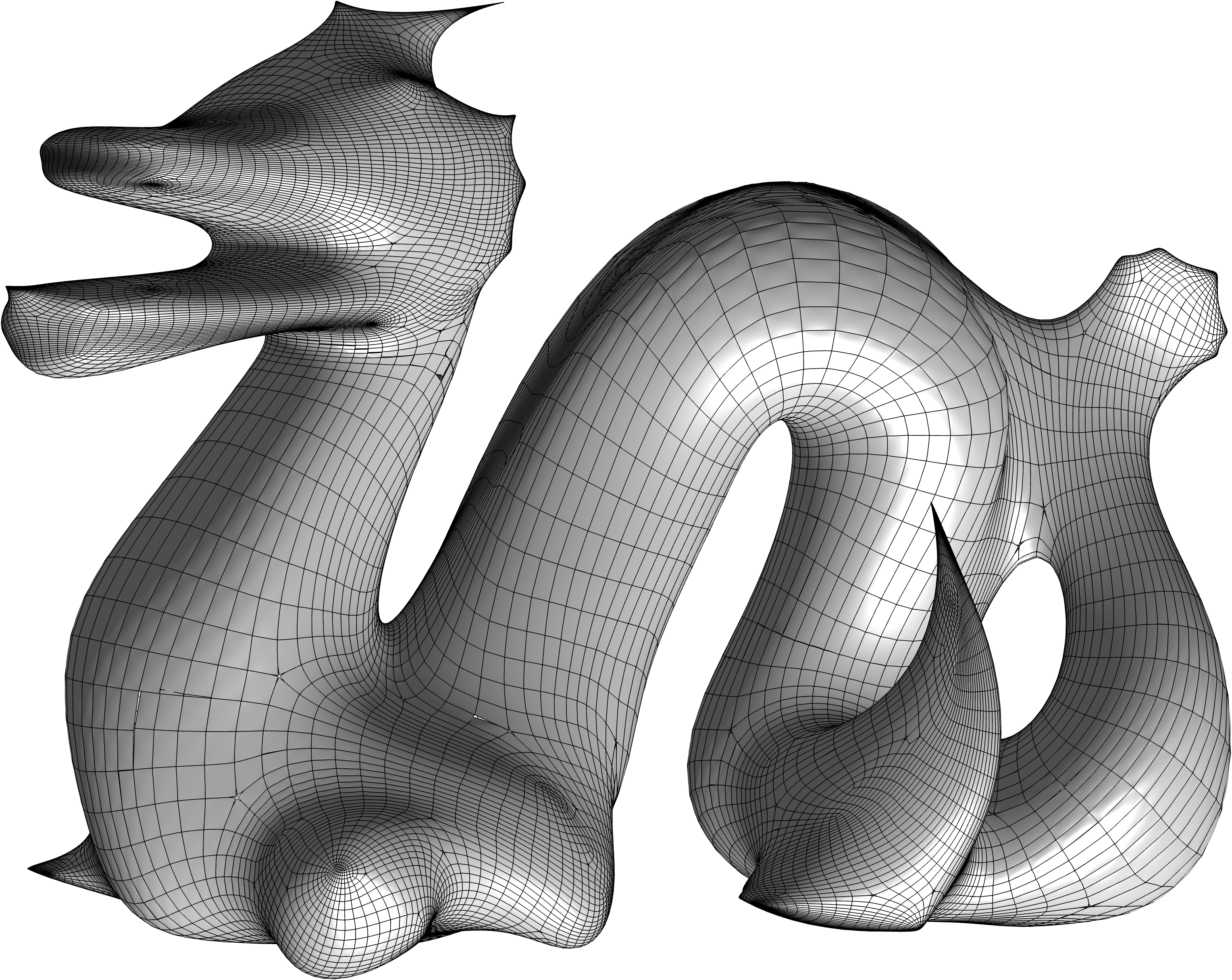}
  \end{subfigure}
  \hfill
  \begin{subfigure}[T]{0.32\textwidth}
    \includegraphics[width=\textwidth, trim={7cm, 2.5cm, 6cm, 2cm}, clip]{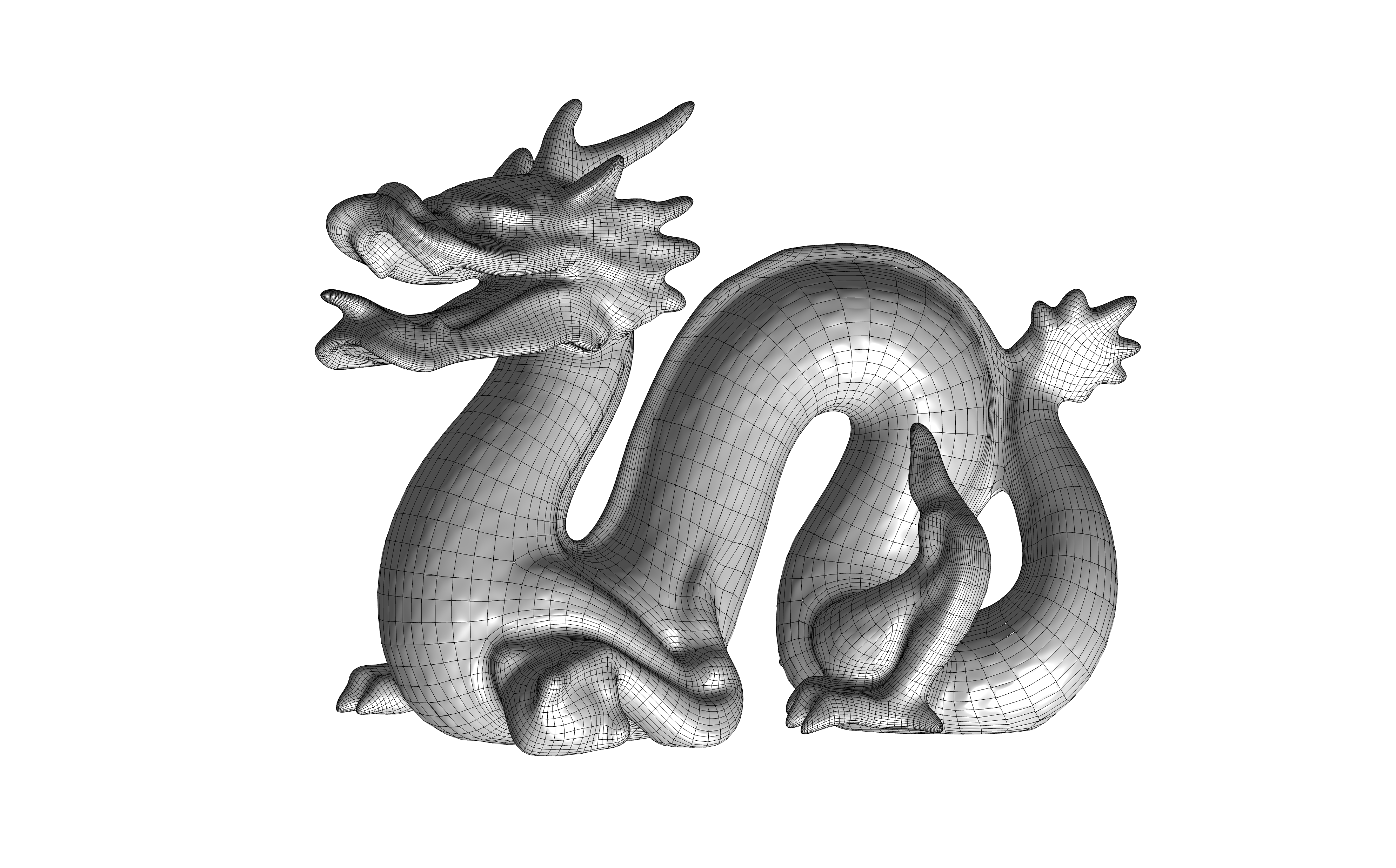}
  \end{subfigure}
  \hfill
  \begin{subfigure}[T]{0.32\textwidth}
    \includegraphics[width=\textwidth, trim={7cm, 2.5cm, 6cm, 2cm}, clip]{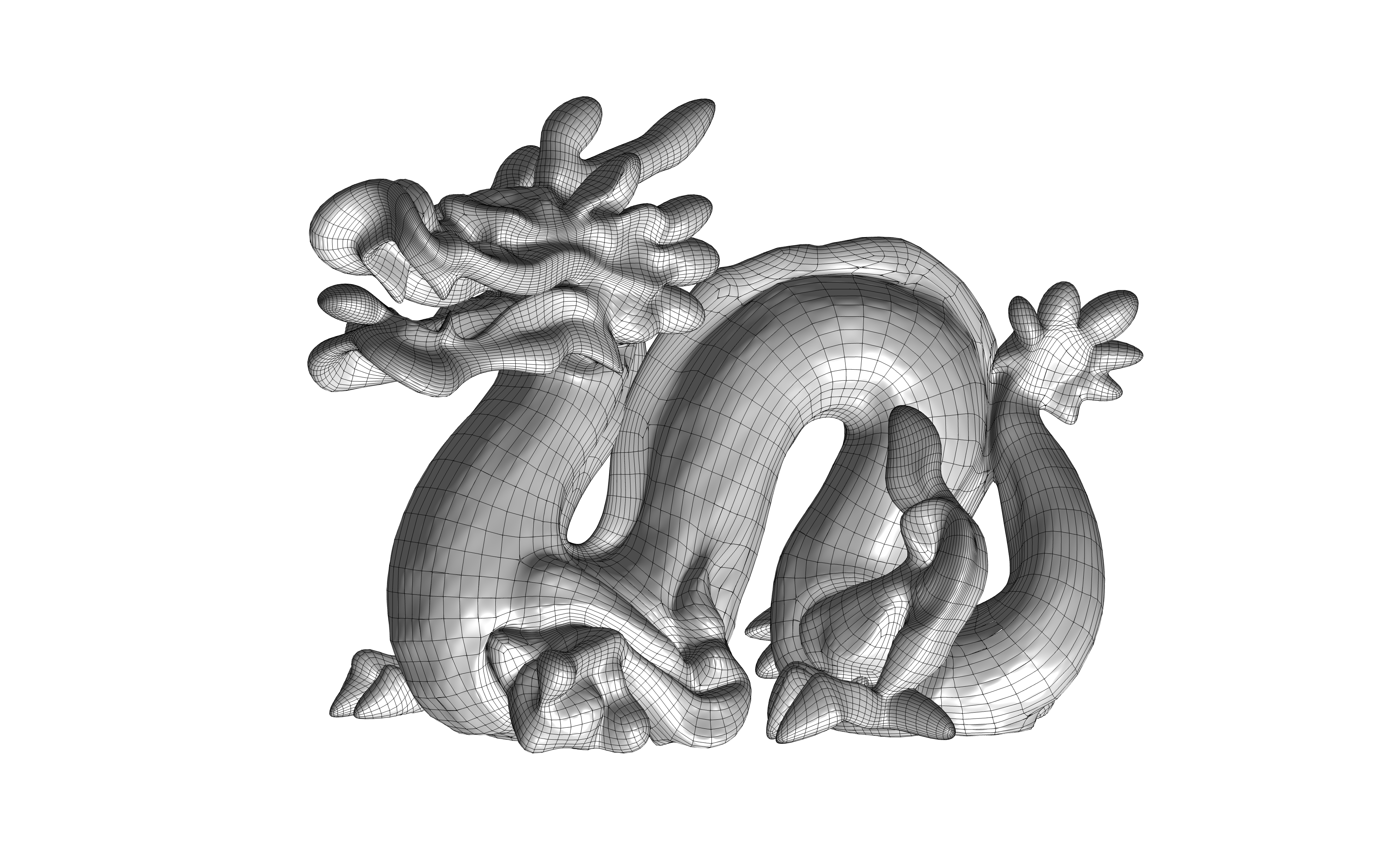}
  \end{subfigure}
  \caption{Filtered mesh with low-pass filter $F(\lambda) =
  \mathds{1}_{\{\lambda < 1000\}}$ (left). Recovered mesh with node coordinates
  $(\tilde{\bm{x}}_1)_i,\dots,(\tilde{\bm{x}}_n)_i$ for $i=1,2,3$ (center).
  Filtered mesh with a Gaussian bump for moderate eigenvalues $F(\lambda) = 1 +
  6400\exp(-(\lambda - 5500)^2/5000^2)$ (right).}
  \label{fig:geometry-processing}
\end{figure}

To illustrate the performance of the BF-MHT in a large-scale example, we use the
spectral collocation-based fast direct solver implemented in the
\texttt{SurfaceFun} package~\cite{fortunato2024high} to compute $m = 5{,}632$
Laplace-Beltrami eigenfunctions on a quadrilateral mesh of a dragon consisting
of 28{,}778 patches with $4 \times 4$ Cheybshev nodes on each patch, for a total
of $n = 460{,}448$ degrees of freedom. As the dense MHT matrix $\mtx{\Phi}$
would require ${\sim}19$ GB of RAM, and the cost of re-orthogonalization of this
large number of eigenvectors within a Krylov solver would be prohibitive, we
divide this computation into multiple spectral bands in a shift-and-invert
fashion. First we compute the surface area $A(\M)$ of the discretized manifold
$\M$, after which Weyl's law provides an estimate of the number of eigenvalues
less than $\lambda$ as
\begin{align}
  N(\lambda) \sim \frac{4\pi \lambda}{A(\M)}.
\end{align}
Due to this asymptotically uniform spectral density, it is straightforward to
choose shifts $\sigma$ so that a roughly equal number of eigenpairs are computed
for each shift with minimal redundant eigenpairs near the boundary between
slices.

For tolerance $\epsilon = 10^{-6}$, the BF-MHT of the dragon mesh requires only
1.33 GB of memory --- a factor of ${\sim}14$ compression over the dense matrix.
For lower tolerance $\epsilon = 10^{-3}$, the BF-MHT achieves a factor of
${\sim}37$ compression, requiring only 533 MB of memory. This lower accuracy,
more compressed MHT would be sufficient for smoothing data or sampling from
random fields (as we discuss in the next section), but we find the higher
accuracy version is required to accurately and stably solve the least squares
problem~\eqref{eq:inverse-MHT}. Figure~\ref{fig:geometry-processing} displays
the recovered mesh with node coordinates
$(\tilde{\bm{x}}_1)_i,\dots,(\tilde{\bm{x}}_n)_i$ for $i=1,2,3$ as well as two
filtered meshes with low-pass and band-bump filters $F$.

\subsection{Gaussian random fields on manifolds}

Gaussian random fields (GRFs)~\cite{williams2006gaussian} are a ubiquitous class
of models in statistics and uncertainty quantification in which observations of
a random process $Y(\bm{x})$ with $\E Y(\bm{x}) \equiv 0$ taken at locations
$\bm{x}_1,\dots,\bm{x}_n$ are modeled as jointly Gaussian $\vct{y} :=
[Y(\bm{x}_1),\dots,Y(\bm{x}_n)] \sim N(\vct{0}, \mtx{\Sigma}_{\bth})$. The
covariance matrix $\mtx{\Sigma}_{\bth}$ is typically defined entry-wise using a
closed-form parametric family of positive definite functions
$(\mtx{\Sigma}_{\bth})_{jk} := K_{\bth}(\bm{x}_j, \bm{x}_k)$. When the
underlying geometry of the observation locations $\{\bm{x}_j\}_{j=1}^n$ is
$\R^d$, a number of parametric families of positive definite functions
$K_{\bth}$ are known in closed form. However, defining positive definite
covariance functions becomes much more challenging on general manifolds $\M$. 

One method for building covariance functions on manifolds is to specify a
positive \textit{spectral density function} $S_{\bth}$, and construct the
covariance function in terms of the Laplace-Beltrami
eigenfunctions~\cite{borovitskiy2020matern,solin2020hilbert}
\begin{align} \label{eq:kernel-representation}
  K_{\bth}(\bm{x}_j, \bm{x}_k) = \sum_{\ell=1}^\infty S_{\bth}(\lambda_\ell) \phi_\ell(\bm{x}_j) \phi_\ell(\bm{x}_k).
\end{align}
As $S_{\bth}$ is positive, $K_{\bth}$ is positive definite by construction.
Truncating the representation~\eqref{eq:kernel-representation} after $m$ terms
yields an approximation to the covariance matrix in terms of the MHT matrix as
$\mtx{\Sigma}_{\bth} \approx \mtx{\Phi}\mtx{S}_{\bth}\mtx{\Phi}^*$, where
$\mtx{S}_{\bth} := \text{diag}(S_{\bth}(\lambda_1), \dots,
S_{\bth}(\lambda_m))$. One can similarly generate approximate samples from the
corresponding process by truncating the Karhunen-Lo\`eve expansion
\begin{equation} \label{eq:KL-expansion}
  Y(\bm{x}_j) = \sum_{\ell=1}^\infty \sqrt{S_{\bth}(\lambda_\ell)} Z_\ell \phi_\ell(\bm{x}_j) \quad \text{for} \ j=1,\dots,n, \quad Z_\ell \iid N(0,1)
\end{equation}
after $m$ terms. This is precisely a MHT $\vct{y} = \mtx{\Phi}
\sqrt{\mtx{\Sigma}_{\bth}} \vct{z}$ where $\vct{z} \sim N(\vct{0}, \vct{I})$.
Therefore, the BF-MHT can be applied directly to rapidly sample from these
processes and compute matrix-vector products with the covariance matrix.
From~\eqref{eq:KL-expansion}, one sees that samples from the GRF inherit the
boundary conditions enforced on the eigenfunctions $\phi_k$. Thus the above
representation is useful in situations where one wishes to constrain the
behavior of the random field near boundaries, e.g. for modeling populations of
aquatic or land animals near shorelines~\cite{solin2019know}.

Figure~\ref{fig:torus-samples} shows samples drawn for various GRF models on two
different geometries. In both cases, we compute $m = \clg{n/32}$ eigenfunctions
and use tolerance $\epsilon = 10^{-3}$. For the torus, we use a spectral
collocation method on a quadrilateral mesh with $n = 51{,}200$ as in
Section~\ref{sec:torus}. For the Lake Superior geometry, we use linear finite
elements on a triangularization with $n = 57{,}091$ nodes, and enforce Dirichlet
boundary conditions. For both these problems, we achieve a factor of ${\sim}10$
compression for the MHT matrix~$\mtx{\Phi}$.


\begin{figure}
  \centering
  \begin{subfigure}[b]{0.32\textwidth}
    \includegraphics[width=\textwidth, trim={4.5cm, 4.5cm, 4cm, 3.4cm}, clip]{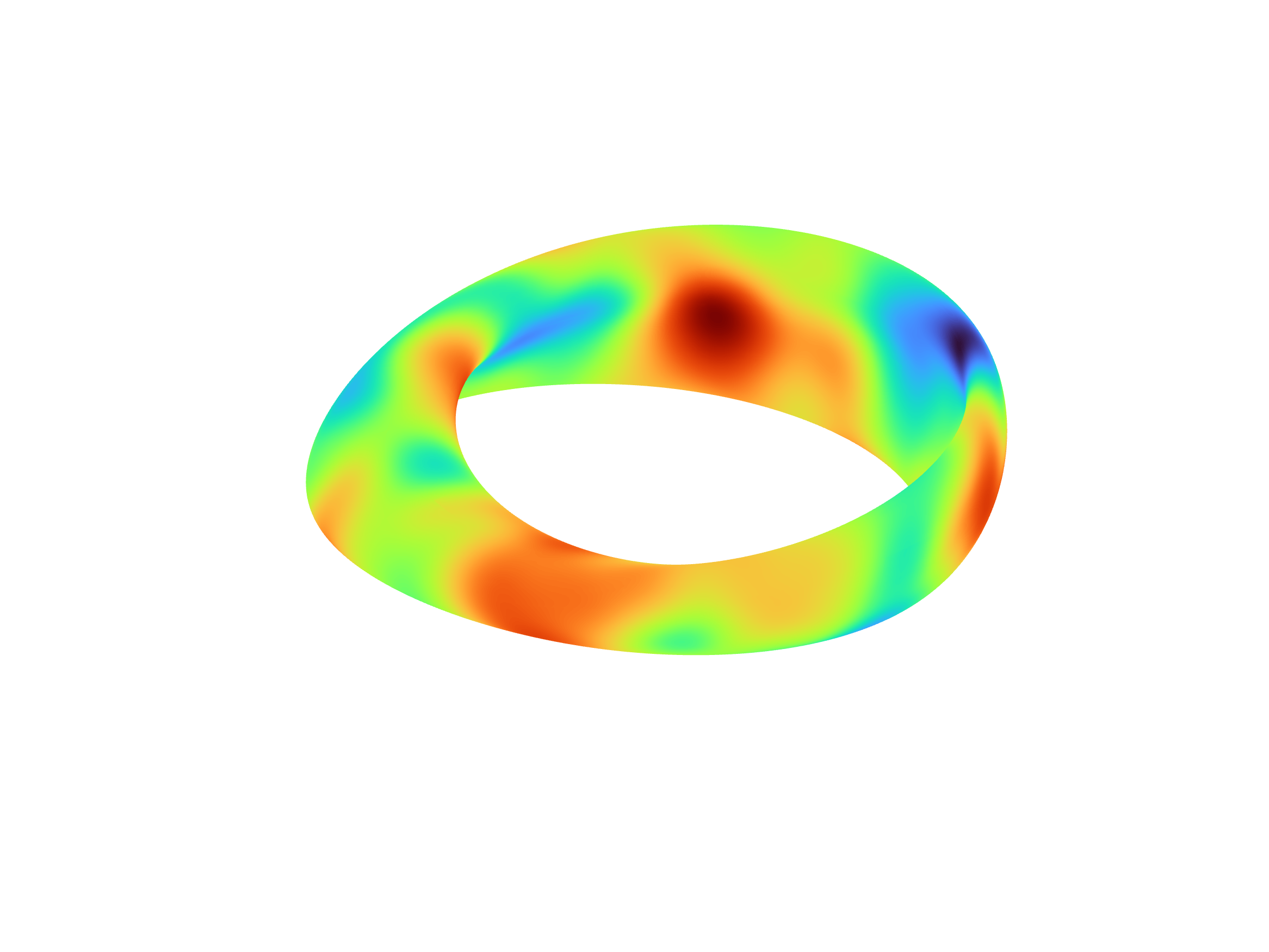}
  \end{subfigure}
  \hfill
  \begin{subfigure}[b]{0.32\textwidth}
    \includegraphics[width=\textwidth, trim={4.5cm, 4.5cm, 4cm, 3.4cm}, clip]{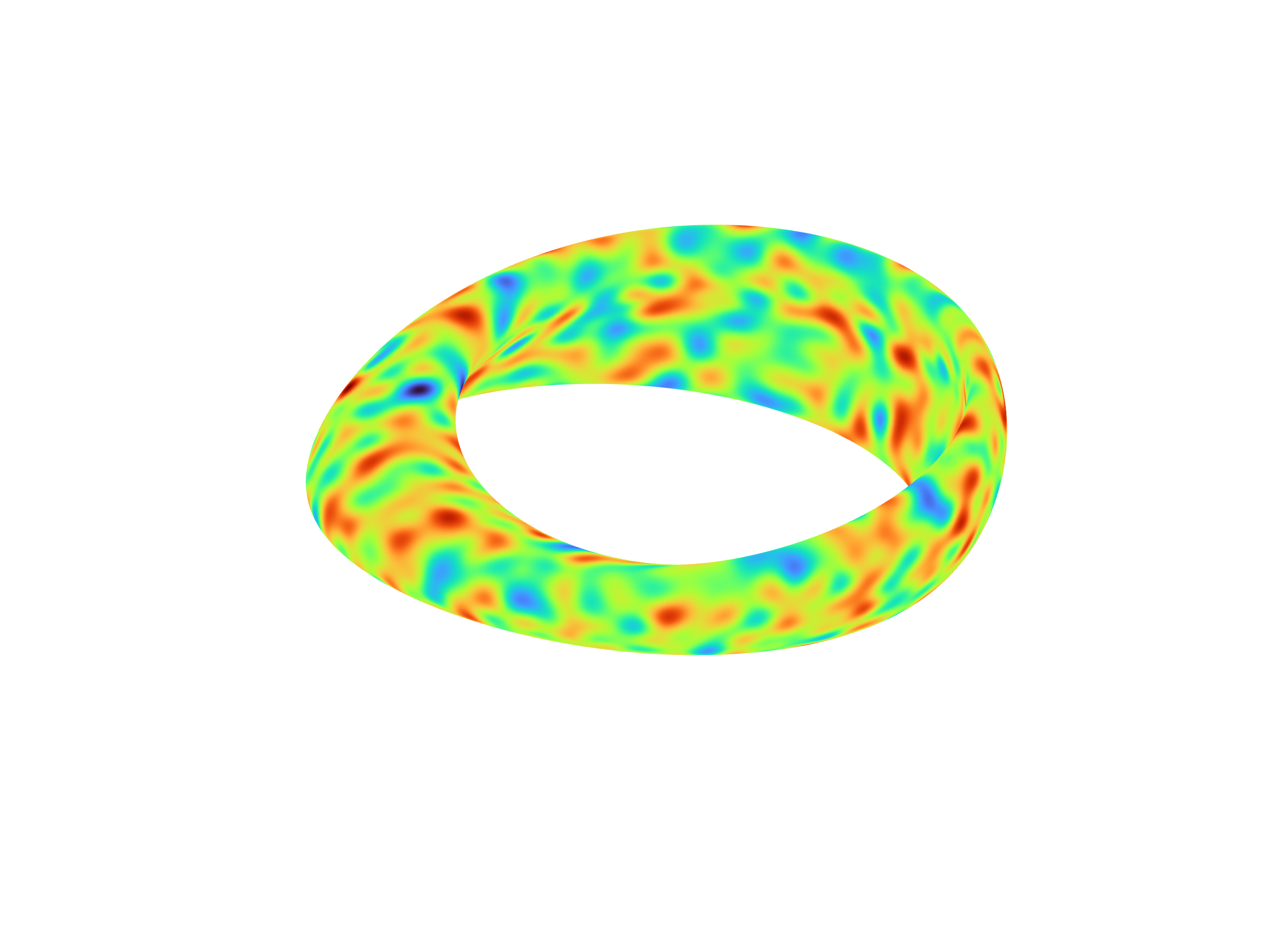}
  \end{subfigure}
  \hfill
  \begin{subfigure}[b]{0.32\textwidth}
    \includegraphics[width=\textwidth, trim={4.5cm, 4.5cm, 4cm, 3.4cm}, clip]{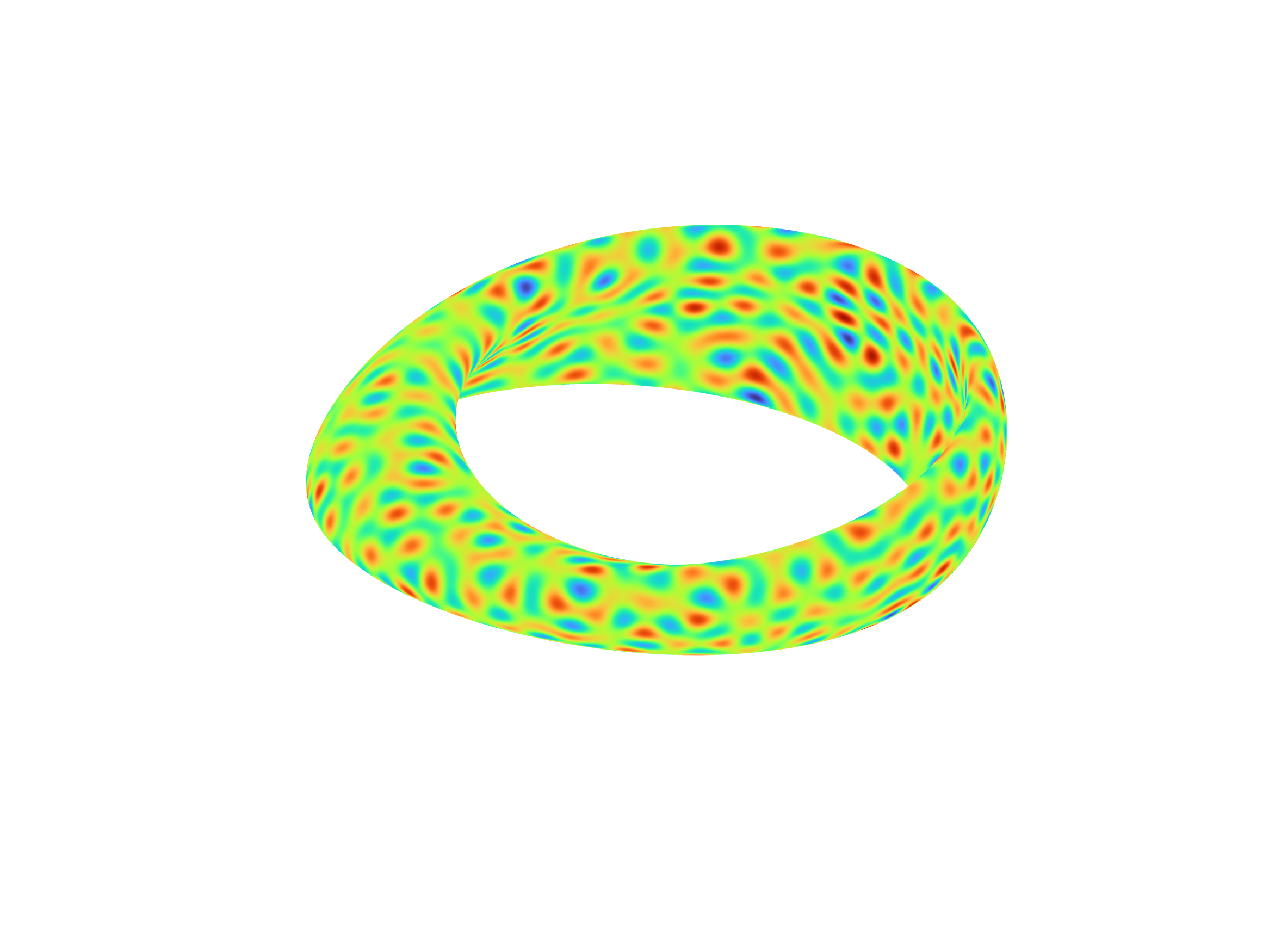}
  \end{subfigure} \\
  \begin{subfigure}[b]{0.32\textwidth}
    \includegraphics[width=\textwidth, trim={4.5cm, 4.5cm, 4cm, 3.4cm}, clip]{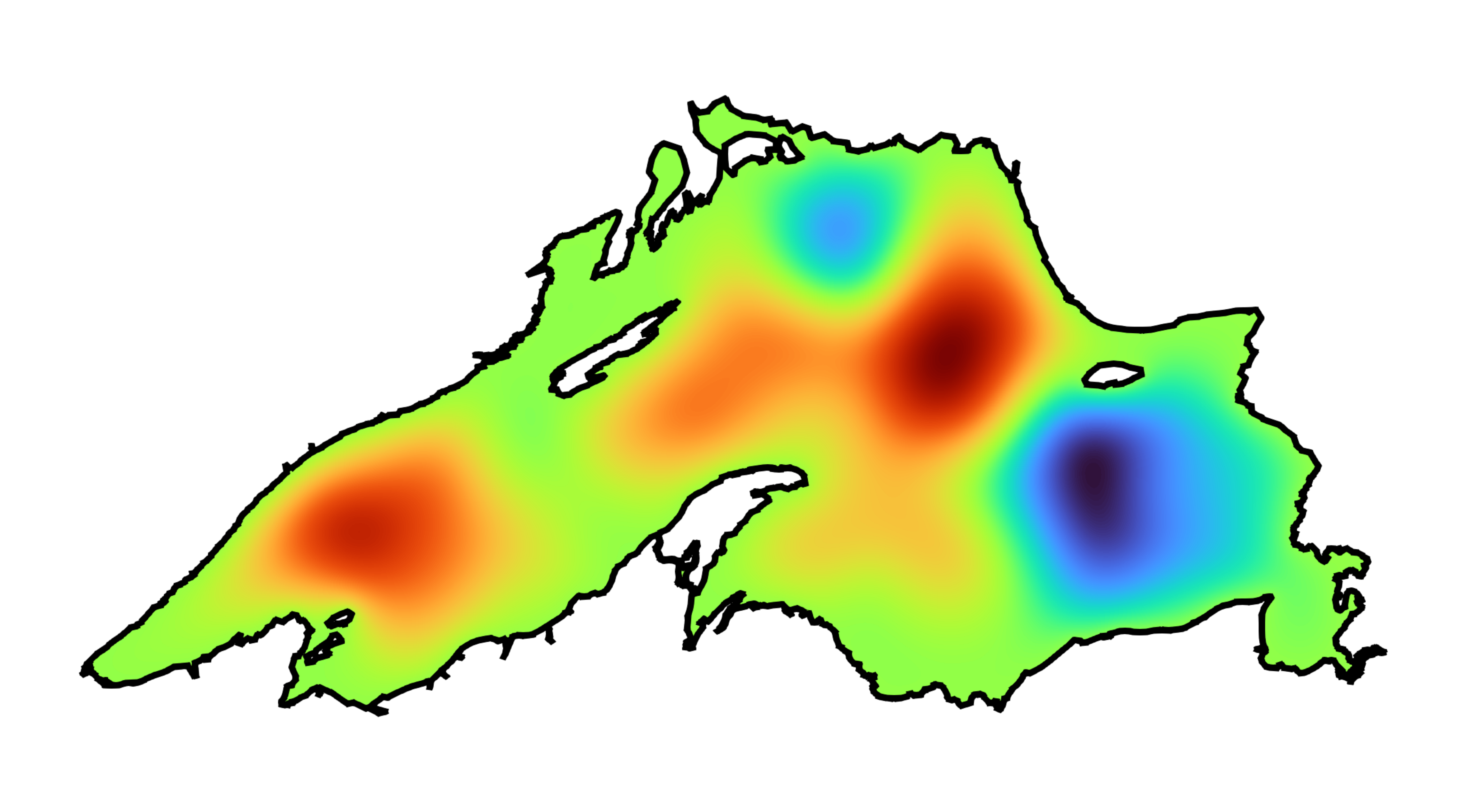}
  \end{subfigure}
  \hfill
  \begin{subfigure}[b]{0.32\textwidth}
    \includegraphics[width=\textwidth, trim={4.5cm, 4.5cm, 4cm, 3.4cm}, clip]{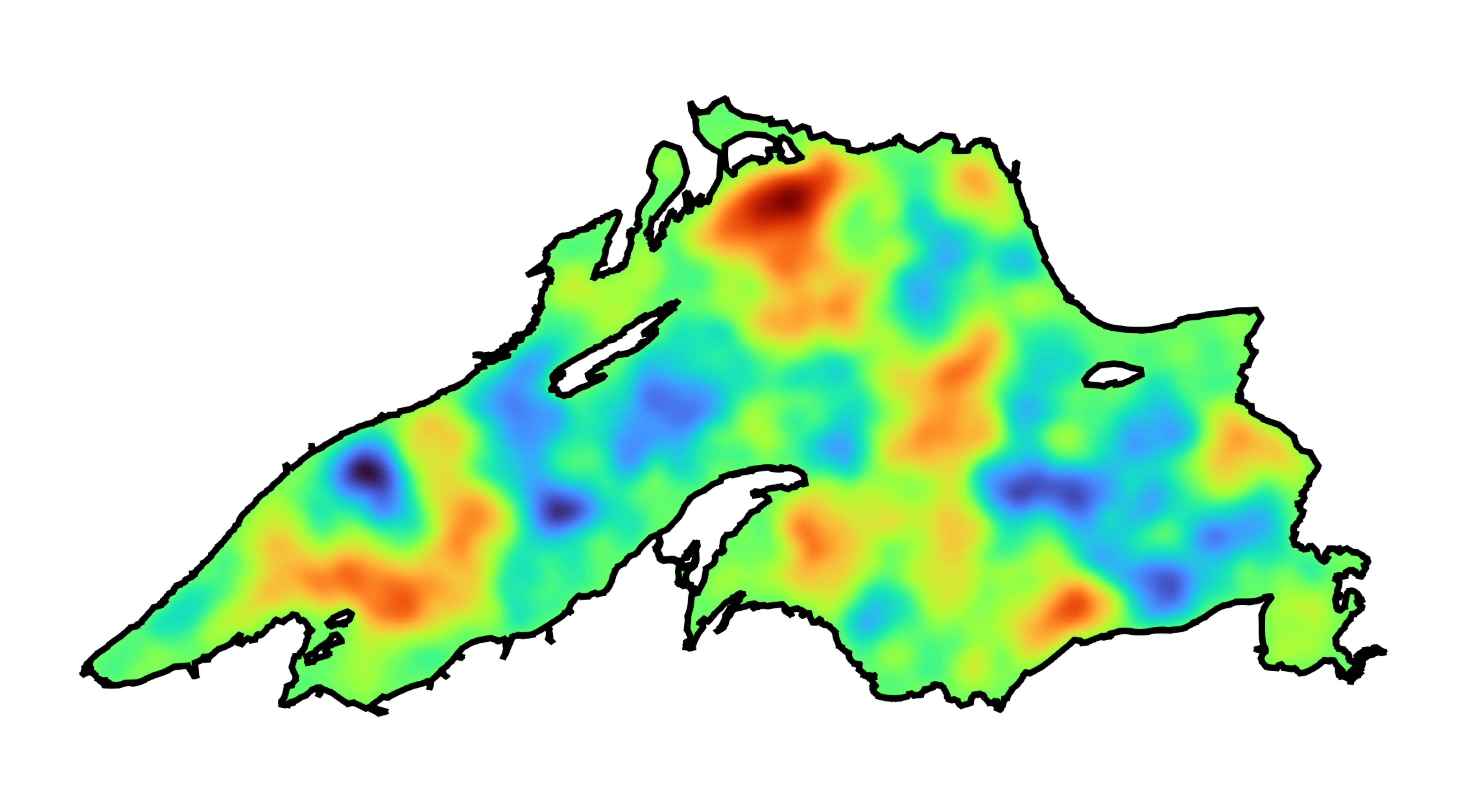}
  \end{subfigure}
  \hfill
  \begin{subfigure}[b]{0.32\textwidth}
    \includegraphics[width=\textwidth, trim={4.5cm, 4.5cm, 4cm, 3.4cm}, clip]{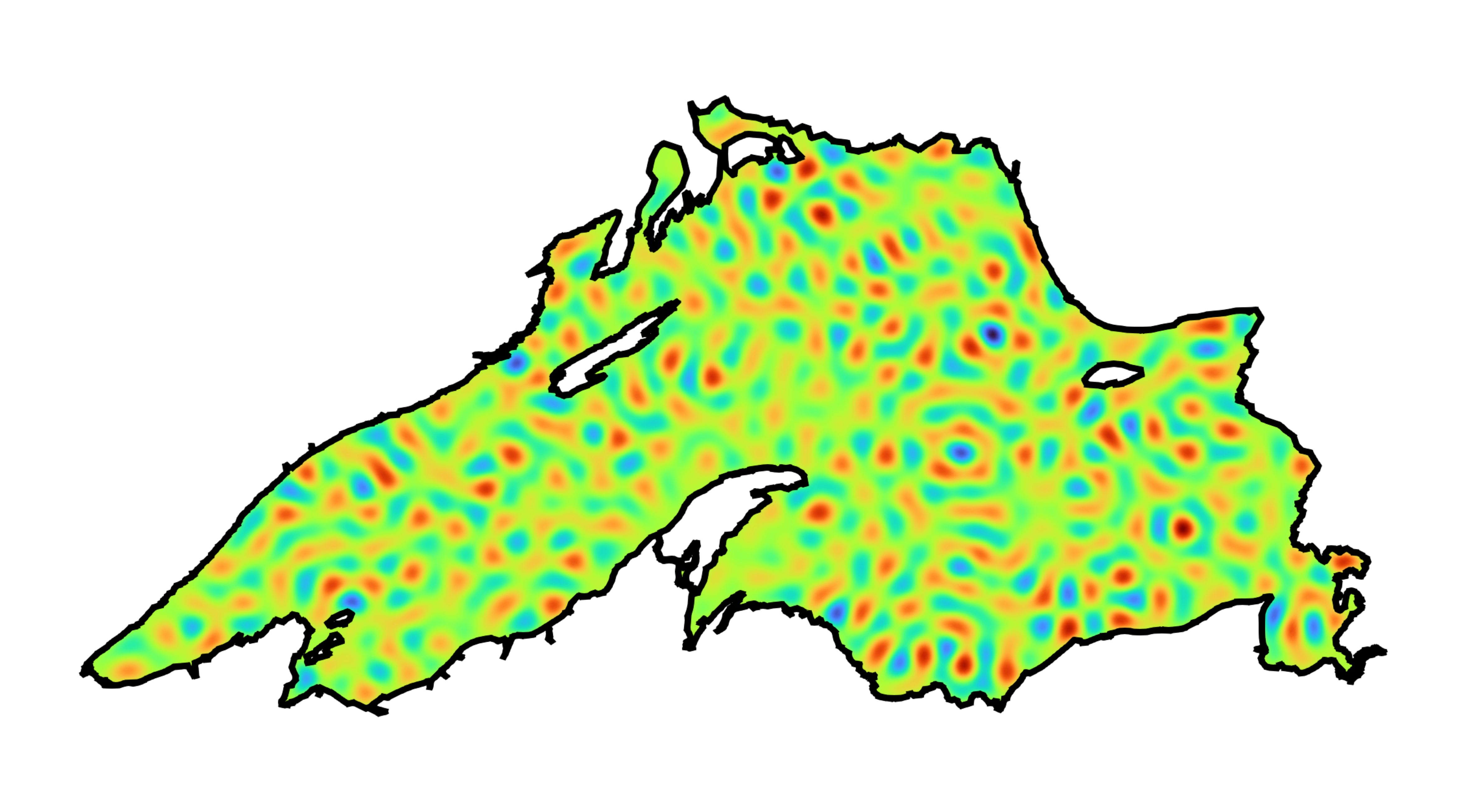}
  \end{subfigure}
  \caption{Samples from GRFs for Mat\'ern spectral densities with $\nu = 3$
  (left) and $\nu = 1$ (center), and for a random surface wave model
  $S_{\bth}(\lambda) = \exp\{-(\lambda - \lambda_0)^2 / \eta^2\}$ (right) on the
  deformed torus quadrilateral mesh and the Lake Superior triangularization with
  Dirichlet boundary conditions.}
  \label{fig:torus-samples}
\end{figure}

\subsection{Laplacian eigenmaps}

Laplacian eigenmaps~\cite{belkin2003laplacian} are a popular machine learning
framework for high-dimensional data which lie approximately on some
lower-dimensional manifold. In this framework, the eigenvectors of a weighted
graph Laplacian are used as approximations to the Laplace-Beltrami
eigenfunctions on the underlying manifold. If the data lie exactly on the
manifold $\M$, then these discrete approximations are known to converge to the
true Laplace-Beltrami eigenfunctions as the number of observations grows, given
a suitably chosen weighting of the graph~\cite{singer2006graph}. It is therefore
natural to consider butterfly compressing the eigenvectors of the weighted graph
Laplacian to construct an accelerated MHT for the mesh-free setting.

A common choice of weighting is the heat kernel, which yields the kernel matrix
\begin{align}
  \mtx{K}_{jk} := e^{-\norm{\vct{x}_j - \vct{x}_k}^2/t},
\end{align}
where $t$ is a parameter which controls the degree of localization. Due to the
exponential decay of the heat kernel, we threshold $\mtx{K}$ and only store
entries greater than some tolerance~$\epsilon$, yielding a sparse approximation
to the full matrix. We then compute the diagonal matrix $\mtx{D}_{jj} :=
\sum_{k=1}^n \mtx{K}_{jk}$ which normalizes the kernel matrix. The desired
Laplacian eigenvectors are obtained by solving the generalized eigenproblem
\begin{align} \label{eq:eigenmap-gen-eig}
  (\mtx{D} - \mtx{K}) \vct{\phi} = \lambda \mtx{D} \vct{\phi}
\end{align} 
which is easily converted by a diagonal rescaling $\vct{v} :=
\mtx{D}^{-\frac{1}{2}} \vct{\phi}$ into the sparse, symmetric standard
eigenproblem
\begin{align} \label{eq:eigenmap-std-eig}
  (\mtx{I} - \mtx{D}^{-\frac{1}{2}} \mtx{K} \mtx{D}^{-\frac{1}{2}}) \vct{v} = \lambda \vct{v}.
\end{align}

To test the effectiveness of the BF-MHT in the Laplacian eigenmaps context, we
perform the following numerical experiment. First, we sample $n$ vertices
$\bm{x}_1, \dots, \bm{x}_n$ uniformly at random from a fine mesh of the human
hand. We then add independent Gaussian noise $\R^3 \ni \bm{\xi}_1, \dots,
\bm{\xi}_n \iid N(\vct{0}, \sigma^2\mtx{I})$ to each sampled vertex to perturb
it away from the true manifold. Next, we form the kernel matrix $\mtx{K}$
corresponding to the noisy points $\bm{x}_j + \bm{\xi}_j$. Following the error
analysis of~\cite{singer2006graph}, we choose $t \propto n^{-1/4}$. We then
zero-out entries of $\mtx{K}$ with absolute values below $\epsilon$ and compute
$m = \clg{n/50}$ eigenvectors of the resulting sparse
eigenproblem~\eqref{eq:eigenmap-std-eig} using a Krylov method. Finally, we
construct a Fiedler tree in physical space and a binary tree in eigenvalue
space, and butterfly compress the eigenvector matrix. We repeat this experiment
for a number of data sizes $n$. Figure~\ref{fig:hand} shows three example
eigenfunctions of the weighted graph Laplacian and the scaling of the BF-MHT
with $n$. We observe the same $\sim n^{3/2}$ memory scaling as in the previous
meshed examples, and achieve a factor of ${\sim}10$ compression for the first $m
= 4000$ eigenfunctions on $n = 200{,}000$ noisy points.

\begin{figure}
  \centering
  \begin{subfigure}[b]{0.22\textwidth}
    \includegraphics[width=\textwidth, trim={3.65cm, 5cm, 6.25cm, 4.7cm}, clip]{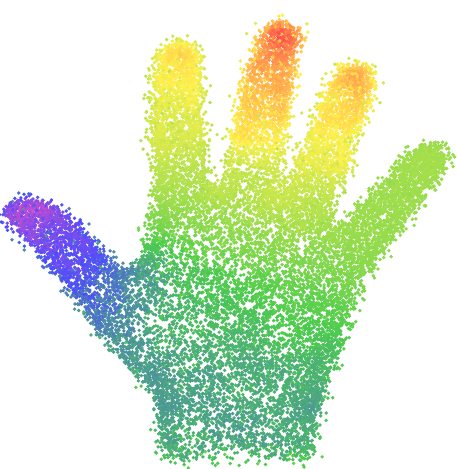}
    \vspace*{0.2cm}
  \end{subfigure}
  \hfill
  \begin{subfigure}[b]{0.22\textwidth}
    \includegraphics[width=\textwidth, trim={3.65cm, 5cm, 6.25cm, 4.7cm}, clip]{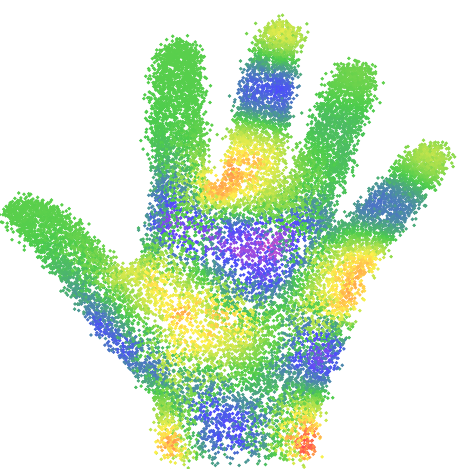}
    \vspace*{0.2cm}
  \end{subfigure}
  \hfill
  \begin{subfigure}[b]{0.22\textwidth}
    \includegraphics[width=\textwidth, trim={3.65cm, 5cm, 6.25cm, 4.7cm}, clip]{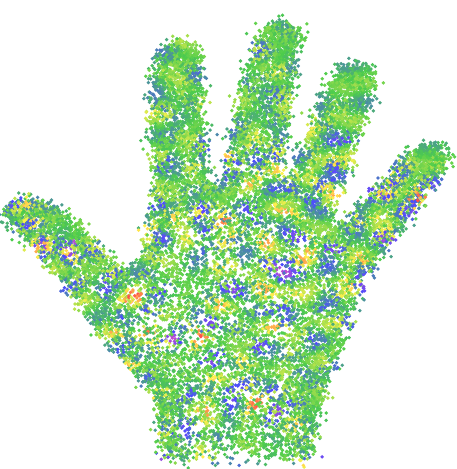}
    \vspace*{0.2cm}
  \end{subfigure}
  \hfill
  \begin{subfigure}[b]{0.3\textwidth}
    \includegraphics[width=\textwidth, trim={0, 0.7cm, 0, 0}, clip]{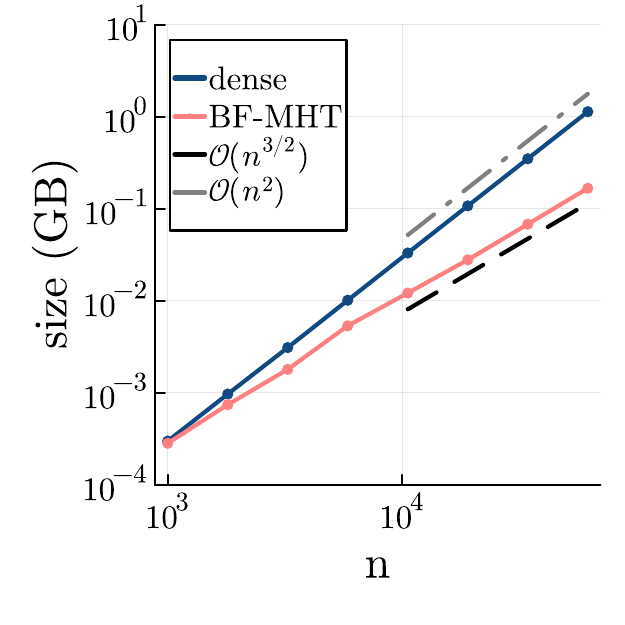}
  \end{subfigure}
  \caption{Three eigenfunctions of the weighted graph Laplacian for noisy data
  on the surface of a hand (left) and scaling of BF-MHT to compress the first $m
  = \clg{\frac{n}{50}}$ eigenfunctions to tolerance $\epsilon = 10^{-3}$
  (right).}
  \label{fig:hand}
\end{figure}

\section{Discussion}

In this work, we introduced a butterfly-accelerated manifold harmonic transform
for rapidly computing linear combinations of Laplace-Beltrami eigenfunctions on
arbitrary compact manifolds. We presented several empirical demonstrations that
show our algorithm requires $\sim n + m^{3/2}$ memory and time to map~$m$
coefficients to~$n$ point values on a selection of compact two-manifolds
in~$\R^3$ for moderate problem sizes. We provided bounds on the rank of the
Fourier transform between an annulus and a disk in two dimensions to prove that
the BF-MHT is $\bO(n + m^{5/3})$ to store and apply in the special case of $\M =
[-\pi, \pi]^2$. Rectifying the empirical results with the associated theoretical
analysis is a subject for future work.

There are several additional interesting directions which are the subject of
ongoing and future work. First, methods for reducing the eigenfunction
computation costs remain of significant practical importance, as this
precomputation requires~$\bO(nm)$ time in general. Of immediate interest are
multiresolution schemes in which low-frequency eigenfunctions are computed on
coarser discretizations, and multiscale approaches~\cite{bennighof2004automated}
in which the global eigenproblem is subdivided into local ones. In addition,
methods for reducing the ranks of compressed blocks by using alternative trees
or by reordering points and frequencies offer the possibility of reducing the
prefactor or even the asymptotic complexity of the BF-MHT. Recent work in this
direction indicates promise~\cite{su2025learning}.

\section*{Acknowledgments}

The authors would like to thank Dan Fortunato for providing technical guidance
and implementing additional requested features within his \texttt{SurfaceFun}
package.

\section*{Competing interests}
The authors report no competing interests.

\bibliographystyle{siamplain}
\bibliography{refs}

\end{document}

%% file: preamble.tex
\newcommand{\E}{\mathbb{E}}    
\newcommand{\Z}{\mathbb{Z}}    
\newcommand{\R}{\mathbb{R}}    
\newcommand{\bbR}{\R}
\newcommand{\C}{\mathbb{C}}    
\newcommand{\M}{\mathcal{M}}   
\newcommand{\bO}{\mathcal{O}}  
\newcommand{\cO}{\bO}
\newcommand{\bth}{\bm{\theta}} 
\newcommand{\abs}[1]{\left|#1\right|}      
\newcommand{\dif}[1]{\mathop{}\!\mathrm{d}{#1}}   
\newcommand{\norm}[2][]{\left\Vert#2\right\Vert_{#1}} 
\newcommand{\clg}[1]{{\left\lceil #1 \right\rceil}}   
\newcommand{\mat}[1]{\begin{bmatrix}#1\end{bmatrix}}  

\renewcommand{\epsilon}{\varepsilon}
\renewcommand{\phi}{\varphi}

\DeclareMathOperator*{\argmin}{arg\,min}   

\newcommand{\pder}[2]{\frac{\partial #1}{\partial #2}} 
\newcommand{\iid}{\mathrel{\overset{\scalebox{0.5}{\text{i.i.d.}}}{\scalebox{1.1}[1]{$\sim$}}}}
\DeclareMathOperator*{\vol}{vol}

\numberwithin{equation}{section}

\newcommand{\mtx}[1]{\bm{\mathsf{#1}}}

\newcommand{\vct}[1]{\bm{\mathsf{#1}}}

\newcommand{\LBO}{\Delta^{}_{\M}}

%% file: figures/BF-levels-tikz.tex
\begin{tikzpicture}[scale=0.27]
  \newcommand{\cs}{0.2}
  %

  \begin{scope}[shift={(0cm, 0cm)}]
    \foreach \x in {1, ..., 8} {
      \draw[fill=black] (\x - 0.5, 8.75) circle (\cs);
    }
    \foreach \x in {1, ..., 4} {
      \def \xnode {2*\x - 1};
      \begin{scope}[on background layer]
        \draw (\xnode - 0.5, 8.75) -- (\xnode, 9.5);
        \draw (\xnode + 0.5, 8.75) -- (\xnode, 9.5);
      \end{scope}
      \draw[fill=white] (\xnode, 9.5) circle (\cs);
    }
    \foreach \x in {1, ..., 2} {
      \def \xnode {4*\x - 2};
      \begin{scope}[on background layer]
        \draw (\xnode - 1, 9.5) -- (\xnode, 10.25);
        \draw (\xnode + 1, 9.5) -- (\xnode, 10.25);
      \end{scope}
      \draw[fill=white] (4*\x - 2, 10.25) circle (\cs);
    }
    \def \xnode {4};
    \begin{scope}[on background layer]
      \draw (\xnode - 2, 10.25) -- (\xnode, 11);
      \draw (\xnode + 2, 10.25) -- (\xnode, 11);
    \end{scope}
    \draw[fill=white] (4, 11) circle (\cs);
    \foreach \y in {1, ..., 8} {
      \draw[fill=white] (-0.75, \y - 0.5) circle (\cs);
    }
    \foreach \y in {1, ..., 4} {
      \def \ynode {2*\y - 1};
      \begin{scope}[on background layer]
        \draw (-0.75, \ynode - 0.5) -- (-1.5, \ynode);
        \draw (-0.75, \ynode + 0.5) -- (-1.5, \ynode);
      \end{scope}
      \draw[fill=white] (-1.5, \ynode) circle (\cs);
    }
    \foreach \y in {1, ..., 2} {
      \def \ynode {4*\y - 2};
      \begin{scope}[on background layer]
        \draw (-1.5, \ynode - 1) -- (-2.25, \ynode);
        \draw (-1.5, \ynode + 1) -- (-2.25, \ynode);
      \end{scope}
      \draw[fill=white] (-2.25, \ynode) circle (\cs);
    }
    \def \ynode {4};
    \begin{scope}[on background layer]
      \draw (-2.25, \ynode - 2) -- (-3, \ynode);
      \draw (-2.25, \ynode + 2) -- (-3, \ynode);
    \end{scope}
    \draw[fill=black] (-3, \ynode) circle (\cs);
    \foreach \y in {1, ..., 8} {
        \draw[fill=white] (\y - 1, 0) rectangle (\y, 8);
    }
    \draw (4, -1) node {Level 0};
  \end{scope}

  %

  \begin{scope}[shift={(12cm, 0cm)}]
    \foreach \x in {1, ..., 8} {
      \draw[fill=white] (\x - 0.5, 8.75) circle (\cs);
    }
    \foreach \x in {1, ..., 4} {
      \def \xnode {2*\x - 1};
      \begin{scope}[on background layer]
        \draw (\xnode - 0.5, 8.75) -- (\xnode, 9.5);
        \draw (\xnode + 0.5, 8.75) -- (\xnode, 9.5);
      \end{scope}
      \draw[fill=black] (\xnode, 9.5) circle (\cs);
    }
    \foreach \x in {1, ..., 2} {
      \def \xnode {4*\x - 2};
      \begin{scope}[on background layer]
        \draw (\xnode - 1, 9.5) -- (\xnode, 10.25);
        \draw (\xnode + 1, 9.5) -- (\xnode, 10.25);
      \end{scope}
      \draw[fill=white] (4*\x - 2, 10.25) circle (\cs);
    }
    \def \xnode {4};
    \begin{scope}[on background layer]
      \draw (\xnode - 2, 10.25) -- (\xnode, 11);
      \draw (\xnode + 2, 10.25) -- (\xnode, 11);
    \end{scope}
    \draw[fill=white] (4, 11) circle (\cs);
    \foreach \y in {1, ..., 8} {
      \draw[fill=white] (-0.75, \y - 0.5) circle (\cs);
    }
    \foreach \y in {1, ..., 4} {
      \def \ynode {2*\y - 1};
      \begin{scope}[on background layer]
        \draw (-0.75, \ynode - 0.5) -- (-1.5, \ynode);
        \draw (-0.75, \ynode + 0.5) -- (-1.5, \ynode);
      \end{scope}
      \draw[fill=white] (-1.5, \ynode) circle (\cs);
    }
    \foreach \y in {1, ..., 2} {
      \def \ynode {4*\y - 2};
      \begin{scope}[on background layer]
        \draw (-1.5, \ynode - 1) -- (-2.25, \ynode);
        \draw (-1.5, \ynode + 1) -- (-2.25, \ynode);
      \end{scope}
      \draw[fill=black] (-2.25, \ynode) circle (\cs);
    }
    \def \ynode {4};
    \begin{scope}[on background layer]
      \draw (-2.25, \ynode - 2) -- (-3, \ynode);
      \draw (-2.25, \ynode + 2) -- (-3, \ynode);
    \end{scope}
    \draw[fill=white] (-3, \ynode) circle (\cs);
    \foreach \x in {1, ..., 2} {
        \foreach \y in {1, ..., 4} {
            \draw[fill=white] (2*\y - 2, 4*\x - 4) rectangle (2*\y, 4*\x);
        }
    }
    \draw (4, -1) node {Level 1};
  \end{scope}

  %

  \begin{scope}[shift={(24cm, 0cm)}]
    \foreach \x in {1, ..., 8} {
      \draw[fill=white] (\x - 0.5, 8.75) circle (\cs);
    }
    \foreach \x in {1, ..., 4} {
      \def \xnode {2*\x - 1};
      \begin{scope}[on background layer]
        \draw (\xnode - 0.5, 8.75) -- (\xnode, 9.5);
        \draw (\xnode + 0.5, 8.75) -- (\xnode, 9.5);
      \end{scope}
      \draw[fill=white] (\xnode, 9.5) circle (\cs);
    }
    \foreach \x in {1, ..., 2} {
      \def \xnode {4*\x - 2};
      \begin{scope}[on background layer]
        \draw (\xnode - 1, 9.5) -- (\xnode, 10.25);
        \draw (\xnode + 1, 9.5) -- (\xnode, 10.25);
      \end{scope}
      \draw[fill=black] (4*\x - 2, 10.25) circle (\cs);
    }
    \def \xnode {4};
    \begin{scope}[on background layer]
      \draw (\xnode - 2, 10.25) -- (\xnode, 11);
      \draw (\xnode + 2, 10.25) -- (\xnode, 11);
    \end{scope}
    \draw[fill=white] (4, 11) circle (\cs);
    \foreach \y in {1, ..., 8} {
      \draw[fill=white] (-0.75, \y - 0.5) circle (\cs);
    }
    \foreach \y in {1, ..., 4} {
      \def \ynode {2*\y - 1};
      \begin{scope}[on background layer]
        \draw (-0.75, \ynode - 0.5) -- (-1.5, \ynode);
        \draw (-0.75, \ynode + 0.5) -- (-1.5, \ynode);
      \end{scope}
      \draw[fill=black] (-1.5, \ynode) circle (\cs);
    }
    \foreach \y in {1, ..., 2} {
      \def \ynode {4*\y - 2};
      \begin{scope}[on background layer]
        \draw (-1.5, \ynode - 1) -- (-2.25, \ynode);
        \draw (-1.5, \ynode + 1) -- (-2.25, \ynode);
      \end{scope}
      \draw[fill=white] (-2.25, \ynode) circle (\cs);
    }
    \def \ynode {4};
    \begin{scope}[on background layer]
      \draw (-2.25, \ynode - 2) -- (-3, \ynode);
      \draw (-2.25, \ynode + 2) -- (-3, \ynode);
    \end{scope}
    \draw[fill=white] (-3, \ynode) circle (\cs);
    \foreach \x in {1, ..., 4} {
        \foreach \y in {1, ..., 2} {
            \draw[fill=white] (4*\y - 4, 2*\x - 2) rectangle (4*\y, 2*\x);
        }
    }
    \draw (4, -1) node {Level 2};
  \end{scope}

  %

  \begin{scope}[shift={(36cm, 0cm)}]
    \foreach \x in {1, ..., 8} {
      \draw[fill=white] (\x - 0.5, 8.75) circle (\cs);
    }
    \foreach \x in {1, ..., 4} {
      \def \xnode {2*\x - 1};
      \begin{scope}[on background layer]
        \draw (\xnode - 0.5, 8.75) -- (\xnode, 9.5);
        \draw (\xnode + 0.5, 8.75) -- (\xnode, 9.5);
      \end{scope}
      \draw[fill=white] (\xnode, 9.5) circle (\cs);
    }
    \foreach \x in {1, ..., 2} {
      \def \xnode {4*\x - 2};
      \begin{scope}[on background layer]
        \draw (\xnode - 1, 9.5) -- (\xnode, 10.25);
        \draw (\xnode + 1, 9.5) -- (\xnode, 10.25);
      \end{scope}
      \draw[fill=white] (4*\x - 2, 10.25) circle (\cs);
    }
    \def \xnode {4};
    \begin{scope}[on background layer]
      \draw (\xnode - 2, 10.25) -- (\xnode, 11);
      \draw (\xnode + 2, 10.25) -- (\xnode, 11);
    \end{scope}
    \draw[fill=black] (4, 11) circle (\cs);
    \foreach \y in {1, ..., 8} {
      \draw[fill=black] (-0.75, \y - 0.5) circle (\cs);
    }
    \foreach \y in {1, ..., 4} {
      \def \ynode {2*\y - 1};
      \begin{scope}[on background layer]
        \draw (-0.75, \ynode - 0.5) -- (-1.5, \ynode);
        \draw (-0.75, \ynode + 0.5) -- (-1.5, \ynode);
      \end{scope}
      \draw[fill=white] (-1.5, \ynode) circle (\cs);
    }
    \foreach \y in {1, ..., 2} {
      \def \ynode {4*\y - 2};
      \begin{scope}[on background layer]
        \draw (-1.5, \ynode - 1) -- (-2.25, \ynode);
        \draw (-1.5, \ynode + 1) -- (-2.25, \ynode);
      \end{scope}
      \draw[fill=white] (-2.25, \ynode) circle (\cs);
    }
    \def \ynode {4};
    \begin{scope}[on background layer]
      \draw (-2.25, \ynode - 2) -- (-3, \ynode);
      \draw (-2.25, \ynode + 2) -- (-3, \ynode);
    \end{scope}
    \draw[fill=white] (-3, \ynode) circle (\cs);
    \foreach \x in {1, ..., 8} {
        \draw[fill=white] (0, \x - 1) rectangle (8, \x);
    }
    \draw (4, -1) node {Level 3};
  \end{scope}
\end{tikzpicture}


%% file: figures/standard-BF-algo.tex
\SetKwFunction{ButterflyFactorization}{ButterflyFactorization}
\Function{\ButterflyFactorization{$\mtx{\Phi}$, $T_x$, $T_\lambda$}}{
    \For{leaf node $\nu$ in $T_\lambda$}{
        $\mtx{U}^{}_{:\nu} \mtx{V}_{:\nu}^* = \mtx{\Phi}(:,\nu)$
    }
    \For{level $\ell = 1,\dots,L$}{
        \For{node $\nu$ at level $\ell$ of $T_\lambda$}{
            Let $c_1$ and $c_2$ be the children of $\nu$ \\
            \For{node $\tau$ at level $L-\ell$ of $T_x$}{
                Let $p$ be the parent of $\tau$ \\
                $\mtx{U}_{\tau\nu}^{} \mtx{R}_{\tau\nu}^* = \mat{\mtx{U}_{pc_1}(\tau, :) & \mtx{U}_{pc_2}(\tau, :)}$ \\
            }
        }
    }
    \Return{$\left\{\hspace*{-2mm}
    \begin{tabular}{l@{}}
    Row basis $\mtx{V}_{:\nu}$ for each leaf node $\nu$ in $T_\lambda$ \\
    Transfer matrix $\mtx{R}_{\tau\nu}$ for each block at levels $\ell=1,\dots,L$ \\
    Column basis $\mtx{U}_{\tau:}$ for each leaf node $\tau$ in $T_x$
    \end{tabular}
    \right\}$
    }
}

%% file: figures/fiedler-algo.tex
\newcommand{\minsize}{\texttt{\upshape \footnotesize min\_size}}
\SetKwFunction{FiedlerTree}{FiedlerTree}
\Function{\FiedlerTree{$\M$, $\{\bm{x}_j\}_{j=1}^n$}}{
    \uIf{$n < \minsize$}{
        \Return{$\{\bm{x}_j\}_{j=1}^n$}
    }
    Compute second eigenfunction $\phi_2$ of $\Delta_{\M}$ \\
    \uIf{discretization of $\M$ consists of patches $P$}{
        $\M_1 := \left\{P \in \M \, : \, \int_P \phi_2(\bm{x}) \dif{\bm{x}} \leq 0 \right\}$ \\ 
    }
    \uElseIf{representation of $\M$ is mesh-free}{
        $\M_1 := \left\{1 \leq j \leq n \, : \, \phi_2(\bm{x}_j) \leq 0 \right\}$ \\ 
    }
    $\M_2 := \M \setminus \M_1$ \\
    $X_1 := \left\{\bm{x}_j \text{ for } j=1,\dots,n \, : \, \bm{x}_j \in \M_1  \right\}$ \\
    $X_2 := \{\bm{x}_j\}_{j=1}^n \setminus X_1$ \\
    \Return{\big\{\FiedlerTree{$\M_1$, $X_1$}, \FiedlerTree{$\M_2$, $X_2$}\big\}}
}

%% file: figures/streaming-BF-algo.tex
\SetKwFunction{StreamingButterflyFactorization}{StreamingButterflyFactorization}
\Function{\StreamingButterflyFactorization{$k \to \mtx{\phi}_k$, $T_x$, $T_\lambda$}}{
    \For{node $\nu$ in a post-order traversal of $T_\lambda$}{
        \uIf{$\nu$ is a leaf node}{
            Compute $\mtx{\Phi}(:,\nu)$ \\
            $\mtx{U}^{}_{:\nu} \mtx{V}_{:\nu}^* = \mtx{\Phi}(:,\nu)$
        }
        \uElse{
            Let $\ell$ be the level of $\nu$ \\
            Let $c_1$ and $c_2$ be the children of $\nu$ \\
            \For{node $\tau$ at level $L-\ell$ of $T_x$}{
                Let $p$ be the parent of $\tau$ \\
                $\mtx{U}_{\tau\nu}^{} \mtx{R}_{\tau\nu}^* = \mat{\mtx{U}_{pc_1}(\tau, :) & \mtx{U}_{pc_2}(\tau, :)}$
            }
            \For{node $p$ at level $L-\ell-1$ of $T_x$}{
                Free memory associated with $\mtx{U}_{pc_1}$ and $\mtx{U}_{pc_2}$
            }
        }
    }
    \Return{$\left\{\hspace*{-2mm}
    \begin{tabular}{l@{}}
    Row basis $\mtx{V}_{:\nu}$ for each leaf node $\nu$ in $T_\lambda$ \\
    Transfer matrix $\mtx{R}_{\tau\nu}$ for each block at levels $\ell=1,\dots,L$ \\
    Column basis $\mtx{U}_{\tau:}$ for each leaf node $\tau$ in $T_x$
    \end{tabular}
    \right\}$
    }
}

%% file: figures/streaming-BF-tikz.tex
\begin{tikzpicture}[scale=0.27]
  \newcommand{\cs}{0.2}
  %

  \begin{scope}[shift={(0cm, 0cm)}]
    \draw[fill=black] (-1.7, 8.5) rectangle (-0.5, 9.7);
    \draw[color=white] (-1.1, 9.1) node {1};
    \draw (4, 12) node {$\lambda$};
    \draw (-4, 4) node {\rotatebox{90}{$\bm{x}$}};
    \foreach \x in {1, ..., 8} {
      \draw[fill=white] (\x - 0.5, 8.75) circle (\cs);
    }
    \foreach \x in {1, ..., 4} {
      \def \xnode {2*\x - 1};
      \begin{scope}[on background layer]
        \draw (\xnode - 0.5, 8.75) -- (\xnode, 9.5);
        \draw (\xnode + 0.5, 8.75) -- (\xnode, 9.5);
      \end{scope}
      \draw[fill=white] (\xnode, 9.5) circle (\cs);
    }
    \foreach \x in {1, ..., 2} {
      \def \xnode {4*\x - 2};
      \begin{scope}[on background layer]
        \draw (\xnode - 1, 9.5) -- (\xnode, 10.25);
        \draw (\xnode + 1, 9.5) -- (\xnode, 10.25);
      \end{scope}
      \draw[fill=white] (4*\x - 2, 10.25) circle (\cs);
    }
    \def \xnode {4};
    \begin{scope}[on background layer]
      \draw (\xnode - 2, 10.25) -- (\xnode, 11);
      \draw (\xnode + 2, 10.25) -- (\xnode, 11);
    \end{scope}
    \draw[fill=white] (4, 11) circle (\cs);
    \foreach \y in {1, ..., 8} {
      \draw[fill=white] (-0.75, \y - 0.5) circle (\cs);
    }
    \foreach \y in {1, ..., 4} {
      \def \ynode {2*\y - 1};
      \begin{scope}[on background layer]
        \draw (-0.75, \ynode - 0.5) -- (-1.5, \ynode);
        \draw (-0.75, \ynode + 0.5) -- (-1.5, \ynode);
      \end{scope}
      \draw[fill=white] (-1.5, \ynode) circle (\cs);
    }
    \foreach \y in {1, ..., 2} {
      \def \ynode {4*\y - 2};
      \begin{scope}[on background layer]
        \draw (-1.5, \ynode - 1) -- (-2.25, \ynode);
        \draw (-1.5, \ynode + 1) -- (-2.25, \ynode);
      \end{scope}
      \draw[fill=white] (-2.25, \ynode) circle (\cs);
    }
    \def \ynode {4};
    \begin{scope}[on background layer]
      \draw (-2.25, \ynode - 2) -- (-3, \ynode);
      \draw (-2.25, \ynode + 2) -- (-3, \ynode);
    \end{scope}
    \draw[fill=white] (-3, \ynode) circle (\cs);
    \draw[fill=gray!30!white] (0, 0) rectangle (8, 8);
  \end{scope}

  %

  \begin{scope}[shift={(12cm, 0cm)}]
    \draw[fill=black] (-1.7, 8.5) rectangle (-0.5, 9.7);
    \draw[color=white] (-1.1, 9.1) node {2};
    \foreach \x in {1, ..., 8} {
      \draw[fill=white] (\x - 0.5, 8.75) circle (\cs);
    }
    \foreach \x in {1, ..., 4} {
      \def \xnode {2*\x - 1};
      \begin{scope}[on background layer]
        \draw (\xnode - 0.5, 8.75) -- (\xnode, 9.5);
        \draw (\xnode + 0.5, 8.75) -- (\xnode, 9.5);
      \end{scope}
      \draw[fill=white] (\xnode, 9.5) circle (\cs);
    }
    \foreach \x in {1, ..., 2} {
      \def \xnode {4*\x - 2};
      \begin{scope}[on background layer]
        \draw (\xnode - 1, 9.5) -- (\xnode, 10.25);
        \draw (\xnode + 1, 9.5) -- (\xnode, 10.25);
      \end{scope}
      \draw[fill=white] (4*\x - 2, 10.25) circle (\cs);
    }
    \def \xnode {4};
    \begin{scope}[on background layer]
      \draw (\xnode - 2, 10.25) -- (\xnode, 11);
      \draw (\xnode + 2, 10.25) -- (\xnode, 11);
    \end{scope}
    \draw[fill=white] (4, 11) circle (\cs);
    \foreach \y in {1, ..., 8} {
      \draw[fill=white] (-0.75, \y - 0.5) circle (\cs);
    }
    \foreach \y in {1, ..., 4} {
      \def \ynode {2*\y - 1};
      \begin{scope}[on background layer]
        \draw (-0.75, \ynode - 0.5) -- (-1.5, \ynode);
        \draw (-0.75, \ynode + 0.5) -- (-1.5, \ynode);
      \end{scope}
      \draw[fill=white] (-1.5, \ynode) circle (\cs);
    }
    \foreach \y in {1, ..., 2} {
      \def \ynode {4*\y - 2};
      \begin{scope}[on background layer]
        \draw (-1.5, \ynode - 1) -- (-2.25, \ynode);
        \draw (-1.5, \ynode + 1) -- (-2.25, \ynode);
      \end{scope}
      \draw[fill=white] (-2.25, \ynode) circle (\cs);
    }
    \def \ynode {4};
    \begin{scope}[on background layer]
      \draw (-2.25, \ynode - 2) -- (-3, \ynode);
      \draw (-2.25, \ynode + 2) -- (-3, \ynode);
    \end{scope}
    \draw[fill=white] (-3, \ynode) circle (\cs);
    \draw[fill=black] (0.5, 8.75) circle (\cs);
    \draw[fill=black] (-3, 4) circle (\cs);
    \draw[fill=gray!30!white] (0, 0) rectangle (8, 8);
    \draw[fill=white] (0, 0) rectangle (1, 8);
  \end{scope}

  %

  \begin{scope}[shift={(24cm, 0cm)}]
    \draw[fill=black] (-1.7, 8.5) rectangle (-0.5, 9.7);
    \draw[color=white] (-1.1, 9.1) node {3};
    \foreach \x in {1, ..., 8} {
      \draw[fill=white] (\x - 0.5, 8.75) circle (\cs);
    }
    \foreach \x in {1, ..., 4} {
      \def \xnode {2*\x - 1};
      \begin{scope}[on background layer]
        \draw (\xnode - 0.5, 8.75) -- (\xnode, 9.5);
        \draw (\xnode + 0.5, 8.75) -- (\xnode, 9.5);
      \end{scope}
      \draw[fill=white] (\xnode, 9.5) circle (\cs);
    }
    \foreach \x in {1, ..., 2} {
      \def \xnode {4*\x - 2};
      \begin{scope}[on background layer]
        \draw (\xnode - 1, 9.5) -- (\xnode, 10.25);
        \draw (\xnode + 1, 9.5) -- (\xnode, 10.25);
      \end{scope}
      \draw[fill=white] (4*\x - 2, 10.25) circle (\cs);
    }
    \def \xnode {4};
    \begin{scope}[on background layer]
      \draw (\xnode - 2, 10.25) -- (\xnode, 11);
      \draw (\xnode + 2, 10.25) -- (\xnode, 11);
    \end{scope}
    \draw[fill=white] (4, 11) circle (\cs);
    \foreach \y in {1, ..., 8} {
      \draw[fill=white] (-0.75, \y - 0.5) circle (\cs);
    }
    \foreach \y in {1, ..., 4} {
      \def \ynode {2*\y - 1};
      \begin{scope}[on background layer]
        \draw (-0.75, \ynode - 0.5) -- (-1.5, \ynode);
        \draw (-0.75, \ynode + 0.5) -- (-1.5, \ynode);
      \end{scope}
      \draw[fill=white] (-1.5, \ynode) circle (\cs);
    }
    \foreach \y in {1, ..., 2} {
      \def \ynode {4*\y - 2};
      \begin{scope}[on background layer]
        \draw (-1.5, \ynode - 1) -- (-2.25, \ynode);
        \draw (-1.5, \ynode + 1) -- (-2.25, \ynode);
      \end{scope}
      \draw[fill=white] (-2.25, \ynode) circle (\cs);
    }
    \def \ynode {4};
    \begin{scope}[on background layer]
      \draw (-2.25, \ynode - 2) -- (-3, \ynode);
      \draw (-2.25, \ynode + 2) -- (-3, \ynode);
    \end{scope}
    \draw[fill=white] (-3, \ynode) circle (\cs);
    \draw[fill=black] (1.5, 8.75) circle (\cs);
    \draw[fill=black] (-3, 4) circle (\cs);
    \draw[fill=gray!30!white] (0, 0) rectangle (8, 8);
    \draw[fill=white] (0, 0) rectangle (1, 8);
    \draw[fill=white] (1, 0) rectangle (2, 8);
  \end{scope}

  %

  \begin{scope}[shift={(36cm, 0cm)}]
    \draw[fill=black] (-1.7, 8.5) rectangle (-0.5, 9.7);
    \draw[color=white] (-1.1, 9.1) node {4};
    \foreach \x in {1, ..., 8} {
      \draw[fill=white] (\x - 0.5, 8.75) circle (\cs);
    }
    \foreach \x in {1, ..., 4} {
      \def \xnode {2*\x - 1};
      \begin{scope}[on background layer]
        \draw (\xnode - 0.5, 8.75) -- (\xnode, 9.5);
        \draw (\xnode + 0.5, 8.75) -- (\xnode, 9.5);
      \end{scope}
      \draw[fill=white] (\xnode, 9.5) circle (\cs);
    }
    \foreach \x in {1, ..., 2} {
      \def \xnode {4*\x - 2};
      \begin{scope}[on background layer]
        \draw (\xnode - 1, 9.5) -- (\xnode, 10.25);
        \draw (\xnode + 1, 9.5) -- (\xnode, 10.25);
      \end{scope}
      \draw[fill=white] (4*\x - 2, 10.25) circle (\cs);
    }
    \def \xnode {4};
    \begin{scope}[on background layer]
      \draw (\xnode - 2, 10.25) -- (\xnode, 11);
      \draw (\xnode + 2, 10.25) -- (\xnode, 11);
    \end{scope}
    \draw[fill=white] (4, 11) circle (\cs);
    \foreach \y in {1, ..., 8} {
      \draw[fill=white] (-0.75, \y - 0.5) circle (\cs);
    }
    \foreach \y in {1, ..., 4} {
      \def \ynode {2*\y - 1};
      \begin{scope}[on background layer]
        \draw (-0.75, \ynode - 0.5) -- (-1.5, \ynode);
        \draw (-0.75, \ynode + 0.5) -- (-1.5, \ynode);
      \end{scope}
      \draw[fill=white] (-1.5, \ynode) circle (\cs);
    }
    \foreach \y in {1, ..., 2} {
      \def \ynode {4*\y - 2};
      \begin{scope}[on background layer]
        \draw (-1.5, \ynode - 1) -- (-2.25, \ynode);
        \draw (-1.5, \ynode + 1) -- (-2.25, \ynode);
      \end{scope}
      \draw[fill=white] (-2.25, \ynode) circle (\cs);
    }
    \def \ynode {4};
    \begin{scope}[on background layer]
      \draw (-2.25, \ynode - 2) -- (-3, \ynode);
      \draw (-2.25, \ynode + 2) -- (-3, \ynode);
    \end{scope}
    \draw[fill=white] (-3, \ynode) circle (\cs);
    \draw[fill=black] (1, 9.5) circle (\cs);
    \draw[fill=black] (-2.25, 6) circle (\cs);
    \draw[fill=black] (-2.25, 2) circle (\cs);
    \draw[fill=gray!30!white] (0, 0) rectangle (8, 8);
    \draw[fill=white] (0, 0) rectangle (2, 4);
    \draw[fill=white] (0, 4) rectangle (2, 8);
  \end{scope}

  %

  \begin{scope}[shift={(0cm, -12cm)}]
    \draw[fill=black] (-1.7, 8.5) rectangle (-0.5, 9.7);
    \draw[color=white] (-1.1, 9.1) node {5};
    \foreach \x in {1, ..., 8} {
      \draw[fill=white] (\x - 0.5, 8.75) circle (\cs);
    }
    \foreach \x in {1, ..., 4} {
      \def \xnode {2*\x - 1};
      \begin{scope}[on background layer]
        \draw (\xnode - 0.5, 8.75) -- (\xnode, 9.5);
        \draw (\xnode + 0.5, 8.75) -- (\xnode, 9.5);
      \end{scope}
      \draw[fill=white] (\xnode, 9.5) circle (\cs);
    }
    \foreach \x in {1, ..., 2} {
      \def \xnode {4*\x - 2};
      \begin{scope}[on background layer]
        \draw (\xnode - 1, 9.5) -- (\xnode, 10.25);
        \draw (\xnode + 1, 9.5) -- (\xnode, 10.25);
      \end{scope}
      \draw[fill=white] (4*\x - 2, 10.25) circle (\cs);
    }
    \def \xnode {4};
    \begin{scope}[on background layer]
      \draw (\xnode - 2, 10.25) -- (\xnode, 11);
      \draw (\xnode + 2, 10.25) -- (\xnode, 11);
    \end{scope}
    \draw[fill=white] (4, 11) circle (\cs);
    \foreach \y in {1, ..., 8} {
      \draw[fill=white] (-0.75, \y - 0.5) circle (\cs);
    }
    \foreach \y in {1, ..., 4} {
      \def \ynode {2*\y - 1};
      \begin{scope}[on background layer]
        \draw (-0.75, \ynode - 0.5) -- (-1.5, \ynode);
        \draw (-0.75, \ynode + 0.5) -- (-1.5, \ynode);
      \end{scope}
      \draw[fill=white] (-1.5, \ynode) circle (\cs);
    }
    \foreach \y in {1, ..., 2} {
      \def \ynode {4*\y - 2};
      \begin{scope}[on background layer]
        \draw (-1.5, \ynode - 1) -- (-2.25, \ynode);
        \draw (-1.5, \ynode + 1) -- (-2.25, \ynode);
      \end{scope}
      \draw[fill=white] (-2.25, \ynode) circle (\cs);
    }
    \def \ynode {4};
    \begin{scope}[on background layer]
      \draw (-2.25, \ynode - 2) -- (-3, \ynode);
      \draw (-2.25, \ynode + 2) -- (-3, \ynode);
    \end{scope}
    \draw[fill=white] (-3, \ynode) circle (\cs);
    \draw[fill=black] (2.5, 8.75) circle (\cs);
    \draw[fill=black] (-3, 4) circle (\cs);
    \draw[fill=gray!30!white] (0, 0) rectangle (8, 8);
    \draw[fill=white] (0, 0) rectangle (2, 4);
    \draw[fill=white] (0, 4) rectangle (2, 8);
    \draw[fill=white] (2, 0) rectangle (3, 8);
  \end{scope}

  %

  \begin{scope}[shift={(12cm, -12cm)}]
    \draw[fill=black] (-1.7, 8.5) rectangle (-0.5, 9.7);
    \draw[color=white] (-1.1, 9.1) node {6};
    \foreach \x in {1, ..., 8} {
      \draw[fill=white] (\x - 0.5, 8.75) circle (\cs);
    }
    \foreach \x in {1, ..., 4} {
      \def \xnode {2*\x - 1};
      \begin{scope}[on background layer]
        \draw (\xnode - 0.5, 8.75) -- (\xnode, 9.5);
        \draw (\xnode + 0.5, 8.75) -- (\xnode, 9.5);
      \end{scope}
      \draw[fill=white] (\xnode, 9.5) circle (\cs);
    }
    \foreach \x in {1, ..., 2} {
      \def \xnode {4*\x - 2};
      \begin{scope}[on background layer]
        \draw (\xnode - 1, 9.5) -- (\xnode, 10.25);
        \draw (\xnode + 1, 9.5) -- (\xnode, 10.25);
      \end{scope}
      \draw[fill=white] (4*\x - 2, 10.25) circle (\cs);
    }
    \def \xnode {4};
    \begin{scope}[on background layer]
      \draw (\xnode - 2, 10.25) -- (\xnode, 11);
      \draw (\xnode + 2, 10.25) -- (\xnode, 11);
    \end{scope}
    \draw[fill=white] (4, 11) circle (\cs);
    \foreach \y in {1, ..., 8} {
      \draw[fill=white] (-0.75, \y - 0.5) circle (\cs);
    }
    \foreach \y in {1, ..., 4} {
      \def \ynode {2*\y - 1};
      \begin{scope}[on background layer]
        \draw (-0.75, \ynode - 0.5) -- (-1.5, \ynode);
        \draw (-0.75, \ynode + 0.5) -- (-1.5, \ynode);
      \end{scope}
      \draw[fill=white] (-1.5, \ynode) circle (\cs);
    }
    \foreach \y in {1, ..., 2} {
      \def \ynode {4*\y - 2};
      \begin{scope}[on background layer]
        \draw (-1.5, \ynode - 1) -- (-2.25, \ynode);
        \draw (-1.5, \ynode + 1) -- (-2.25, \ynode);
      \end{scope}
      \draw[fill=white] (-2.25, \ynode) circle (\cs);
    }
    \def \ynode {4};
    \begin{scope}[on background layer]
      \draw (-2.25, \ynode - 2) -- (-3, \ynode);
      \draw (-2.25, \ynode + 2) -- (-3, \ynode);
    \end{scope}
    \draw[fill=white] (-3, \ynode) circle (\cs);
    \draw[fill=black] (3.5, 8.75) circle (\cs);
    \draw[fill=black] (-3, 4) circle (\cs);
    \draw[fill=gray!30!white] (0, 0) rectangle (8, 8);
    \draw[fill=white] (0, 0) rectangle (2, 4);
    \draw[fill=white] (0, 4) rectangle (2, 8);
    \draw[fill=white] (2, 0) rectangle (3, 8);
    \draw[fill=white] (3, 0) rectangle (4, 8);
  \end{scope}

  %

  \begin{scope}[shift={(24cm, -12cm)}]
    \draw[fill=black] (-1.7, 8.5) rectangle (-0.5, 9.7);
    \draw[color=white] (-1.1, 9.1) node {7};
    \foreach \x in {1, ..., 8} {
      \draw[fill=white] (\x - 0.5, 8.75) circle (\cs);
    }
    \foreach \x in {1, ..., 4} {
      \def \xnode {2*\x - 1};
      \begin{scope}[on background layer]
        \draw (\xnode - 0.5, 8.75) -- (\xnode, 9.5);
        \draw (\xnode + 0.5, 8.75) -- (\xnode, 9.5);
      \end{scope}
      \draw[fill=white] (\xnode, 9.5) circle (\cs);
    }
    \foreach \x in {1, ..., 2} {
      \def \xnode {4*\x - 2};
      \begin{scope}[on background layer]
        \draw (\xnode - 1, 9.5) -- (\xnode, 10.25);
        \draw (\xnode + 1, 9.5) -- (\xnode, 10.25);
      \end{scope}
      \draw[fill=white] (4*\x - 2, 10.25) circle (\cs);
    }
    \def \xnode {4};
    \begin{scope}[on background layer]
      \draw (\xnode - 2, 10.25) -- (\xnode, 11);
      \draw (\xnode + 2, 10.25) -- (\xnode, 11);
    \end{scope}
    \draw[fill=white] (4, 11) circle (\cs);
    \foreach \y in {1, ..., 8} {
      \draw[fill=white] (-0.75, \y - 0.5) circle (\cs);
    }
    \foreach \y in {1, ..., 4} {
      \def \ynode {2*\y - 1};
      \begin{scope}[on background layer]
        \draw (-0.75, \ynode - 0.5) -- (-1.5, \ynode);
        \draw (-0.75, \ynode + 0.5) -- (-1.5, \ynode);
      \end{scope}
      \draw[fill=white] (-1.5, \ynode) circle (\cs);
    }
    \foreach \y in {1, ..., 2} {
      \def \ynode {4*\y - 2};
      \begin{scope}[on background layer]
        \draw (-1.5, \ynode - 1) -- (-2.25, \ynode);
        \draw (-1.5, \ynode + 1) -- (-2.25, \ynode);
      \end{scope}
      \draw[fill=white] (-2.25, \ynode) circle (\cs);
    }
    \def \ynode {4};
    \begin{scope}[on background layer]
      \draw (-2.25, \ynode - 2) -- (-3, \ynode);
      \draw (-2.25, \ynode + 2) -- (-3, \ynode);
    \end{scope}
    \draw[fill=white] (-3, \ynode) circle (\cs);
    \draw[fill=black] (3, 9.5) circle (\cs);
    \draw[fill=black] (-2.25, 6) circle (\cs);
    \draw[fill=black] (-2.25, 2) circle (\cs);
    \draw[fill=gray!30!white] (0, 0) rectangle (8, 8);
    \draw[fill=white] (0, 0) rectangle (2, 4);
    \draw[fill=white] (2, 0) rectangle (4, 4);
    \draw[fill=white] (0, 4) rectangle (2, 8);
    \draw[fill=white] (2, 4) rectangle (4, 8);
  \end{scope}

  %

  \begin{scope}[shift={(36cm, -12cm)}]
    \draw[fill=black] (-1.7, 8.5) rectangle (-0.5, 9.7);
    \draw[color=white] (-1.1, 9.1) node {8};
    \foreach \x in {1, ..., 8} {
      \draw[fill=white] (\x - 0.5, 8.75) circle (\cs);
    }
    \foreach \x in {1, ..., 4} {
      \def \xnode {2*\x - 1};
      \begin{scope}[on background layer]
        \draw (\xnode - 0.5, 8.75) -- (\xnode, 9.5);
        \draw (\xnode + 0.5, 8.75) -- (\xnode, 9.5);
      \end{scope}
      \draw[fill=white] (\xnode, 9.5) circle (\cs);
    }
    \foreach \x in {1, ..., 2} {
      \def \xnode {4*\x - 2};
      \begin{scope}[on background layer]
        \draw (\xnode - 1, 9.5) -- (\xnode, 10.25);
        \draw (\xnode + 1, 9.5) -- (\xnode, 10.25);
      \end{scope}
      \draw[fill=white] (4*\x - 2, 10.25) circle (\cs);
    }
    \def \xnode {4};
    \begin{scope}[on background layer]
      \draw (\xnode - 2, 10.25) -- (\xnode, 11);
      \draw (\xnode + 2, 10.25) -- (\xnode, 11);
    \end{scope}
    \draw[fill=white] (4, 11) circle (\cs);
    \foreach \y in {1, ..., 8} {
      \draw[fill=white] (-0.75, \y - 0.5) circle (\cs);
    }
    \foreach \y in {1, ..., 4} {
      \def \ynode {2*\y - 1};
      \begin{scope}[on background layer]
        \draw (-0.75, \ynode - 0.5) -- (-1.5, \ynode);
        \draw (-0.75, \ynode + 0.5) -- (-1.5, \ynode);
      \end{scope}
      \draw[fill=white] (-1.5, \ynode) circle (\cs);
    }
    \foreach \y in {1, ..., 2} {
      \def \ynode {4*\y - 2};
      \begin{scope}[on background layer]
        \draw (-1.5, \ynode - 1) -- (-2.25, \ynode);
        \draw (-1.5, \ynode + 1) -- (-2.25, \ynode);
      \end{scope}
      \draw[fill=white] (-2.25, \ynode) circle (\cs);
    }
    \def \ynode {4};
    \begin{scope}[on background layer]
      \draw (-2.25, \ynode - 2) -- (-3, \ynode);
      \draw (-2.25, \ynode + 2) -- (-3, \ynode);
    \end{scope}
    \draw[fill=white] (-3, \ynode) circle (\cs);
    \draw[fill=black] (2, 10.25) circle (\cs);
    \draw[fill=black] (-1.5, 7) circle (\cs);
    \draw[fill=black] (-1.5, 5) circle (\cs);
    \draw[fill=black] (-1.5, 3) circle (\cs);
    \draw[fill=black] (-1.5, 1) circle (\cs);
    \draw[fill=gray!30!white] (0, 0) rectangle (8, 8);
    \draw[fill=white] (0, 0) rectangle (4, 2);
    \draw[fill=white] (0, 2) rectangle (4, 4);
    \draw[fill=white] (0, 4) rectangle (4, 6);
    \draw[fill=white] (0, 6) rectangle (4, 8);
  \end{scope}
\end{tikzpicture}


%% file: figures/annulus-to-disk-tikz.tex
\begin{tikzpicture}[scale=0.2]
  \newcommand{\cs}{0.2}

  \begin{scope}[shift={(0cm, 0cm)}]
    \draw[fill=blue!30!white] (0,0) circle (7.5);
    \draw[fill=white] (0,0) circle (5.5);

    \draw (5.4,0) node[anchor=south east] {\small $\sqrt{\rho_{k-1}}$}; 
    \draw (6.8,0) node[anchor=south west] {\small $\sqrt{\rho_{k}}$}; 

    \foreach \x/\y in {
        6/0, 7/0, 
        6/1, 7/1,
        6/2, 7/2,
        5/3, 6/3,
        4/4, 5/4, 6/4,
        3/5, 4/5, 5/5,
        1/6, 2/6, 3/6, 4/6,
        1/7, 2/7}{
        \draw[blue!30!white,fill=black] (\x,\y) circle (\cs);
        \draw[blue!30!white,fill=black] (-1*\y,\x) circle (\cs);
        \draw[blue!30!white,fill=black] (-1*\x,-1*\y) circle (\cs);
        \draw[blue!30!white,fill=black] (\y,-1*\x) circle (\cs);
    }

    \draw[<->] (-10,0) -- (10,0);
    \draw[<->] (0,-10) -- (0,10);
    \draw (-10,0) node[anchor=south] {$\omega_1$}; 
    \draw (0,10)  node[anchor=west]  {$\omega_2$};
  \end{scope}

  \begin{scope}[shift={(15cm, 0cm)}]
    \fill[blue!30!white] (5.5,-0.6) rectangle (7.5,0.6);
    \draw (5.5,-1) -- (5.5,1);
    \draw (7.5,-1) -- (7.5,1);

    \draw (5.5,0) node[anchor=south east] {$\rho_{k-1}$}; 
    \draw (7.5,0) node[anchor=south west] {$\rho_{k}$}; 

    \foreach \x/\y in {
        6/0, 7/0, 
        6/1, 7/1,
        6/2, 7/2,
        5/3, 6/3,
        4/4, 5/4, 6/4,
        3/5, 4/5, 5/5,
        1/6, 2/6, 3/6, 4/6,
        1/7, 2/7}{
        \pgfmathparse{sqrt(\x^2 + \y^2)}
        \draw[blue!30!white,fill=black] (\pgfmathresult, 0) circle (\cs);
    }

    \draw[->] (0,0) -- (10,0);
    \draw (0,-0.5) -- (0,0.5);
    \draw (0,-0.3)  node[anchor=north] {0};
    \draw (10,-0.1) node[anchor=north] {$\lambda$};
  \end{scope}

  \begin{scope}[shift={(40cm, 0cm)}]
    \draw[fill=red!30!white] (-4,2) circle (2.8284271247);
    \draw[->] (-4,2) -- (-4-1.4142135624,2+2.4494897428);
    \draw[fill] (-4,2) circle (0.1);
    \draw (-3,1) node {$\bm{c}_j$}; 
    \draw (-3.8,3.8) node {$R$};

    \draw (-8,-8) rectangle (8,8);

    \draw (8,0) node[anchor=south west] {$\pi$}; 

    \draw[<->] (-10,0) -- (10,0);
    \draw[<->] (0,-10) -- (0,10);
    \draw (-10,0) node[anchor=south] {$x_1$}; 
    \draw (0,10)  node[anchor=west]  {$x_2$};
  \end{scope}
  
  \draw[->,thick,dashed] (6,6) to [bend left=40] (21,2);
  \draw (11.5, 9) node {$\norm{\bm{\omega}}^2 = \lambda$};

  \draw[->,thick] (22,-2) to [bend right=60] (34,-2);
  \draw (28, -6.5) node {$\mtx{\Phi}(\tau,\nu)$};

\end{tikzpicture}